\documentclass[10pt]{amsart}

\usepackage{amsmath}
\usepackage{amssymb}
\usepackage{bm}
\usepackage{graphicx}
\usepackage{psfrag}
\usepackage{color}
\usepackage{xcolor}
\usepackage{listings}

\definecolor{keywordcolor}{rgb}{0.2, 0.2, 0.75}
\definecolor{stringcolor}{rgb}{0.0, 0.5, 0.0}
\definecolor{commentcolor}{rgb}{0.5, 0.5, 0.5}
\definecolor{backgroundcolor}{rgb}{0.95, 0.95, 0.95}
\definecolor{numbercolor}{rgb}{0.3, 0.3, 0.3}

\lstdefinestyle{pythonstyle}{
	language=Python,
	backgroundcolor=\color{backgroundcolor},
	basicstyle=\ttfamily\small,
	keywordstyle=\color{keywordcolor}\bfseries,
	stringstyle=\color{stringcolor},
	commentstyle=\color{commentcolor}\itshape,
	numberstyle=\tiny\color{numbercolor},
	numbers=left,
	numbersep=10pt,
	showstringspaces=false,
	breaklines=true,
	frame=single,
	tabsize=4,
	captionpos=b
}

\definecolor{red}{rgb}{1,0,0}
\definecolor{felix}{rgb}{1,0,0}

\definecolor{pink}{rgb}{1,0,0.4}
\definecolor{marly}{rgb}{1,0,0.4}

\definecolor{anna}{HTML}{d595fc}

\definecolor{jason}{HTML}{4a8e1f}

\usepackage{hyperref}
\hypersetup{colorlinks=true, linkcolor=blue, citecolor=magenta, urlcolor=green}
\usepackage{url}
\usepackage{algpseudocode}
\usepackage{fancyhdr}
\usepackage{mathtools}
\usepackage{tikz-cd}
\usepackage{xy}
\input xy
\xyoption{all}
\usepackage{stmaryrd}
\usepackage{calrsfs}

\newtheorem{theorem}{Theorem}[section]
\newtheorem{thm}[theorem]{Theorem}
\newtheorem{lemma}[theorem]{Lemma}
\newtheorem{lem}[theorem]{Lemma}
\newtheorem{proposition}[theorem]{Proposition}
\newtheorem{prop}[theorem]{Proposition}
\newtheorem{corollary}[theorem]{Corollary}
\newtheorem{cor}[theorem]{Corollary}

\theoremstyle{definition}

\newtheorem{defn}[theorem]{Definition}
\newtheorem{example}[theorem]{Example}
\newtheorem{exam}[theorem]{Example}

\newtheorem{question}[theorem]{Question}
\newtheorem{quest}[theorem]{Question}

\theoremstyle{remark}

\numberwithin{equation}{section}

\newcommand{\aaa}{\mathbb{A}}
\newcommand{\cc}{\mathbb{C}}
\newcommand{\ff}{\mathbb{F}}
\newcommand{\nn}{\mathbb{N}}

\newcommand{\pp}{\mathbb{P}}

\newcommand{\qq}{\mathbb{Q}}

\newcommand{\rr}{\mathbb{R}}
\newcommand{\zz}{\mathbb{Z}}

\providecommand\ldb{\llbracket}
\providecommand\rdb{\rrbracket}

\newcommand{\gp}{\mathcal{G}}

\newcommand{\num}{\mathsf{n}}
\newcommand{\den}{\mathsf{d}}



\begin{document}
	
\mbox{}
\title{Factorizations in rational monogenic semidomains}

\author{Anna Deng}
\address{CMI\\MIT\\Cambridge, MA 02139}
\email{annadeng08@gmail.com}

\author{Felix Gotti}
\address{Department of Mathematics\\MIT\\Cambridge, MA 02139}
\email{fgotti@mit.edu}

\author{Jason Zeng}
\address{CMI\\MIT\\Cambridge, MA 02139}
\email{jasonzeng124@gmail.com}

\date{\today}

\begin{abstract}
	For $\rho \in \mathbb{C}$, the monogenic semidomain generated by $\rho$ is the smallest subsemiring $S_\rho$ of the complex field $\mathbb{C}$ containing $\rho$. We initiate a systematic study of the arithmetic of the semidomain $S_q$ for rational parameters $q$, focusing on its fundamental factorization-theoretic properties. After some preliminaries, we introduce and investigate the monoid of technical fractions $T_q$, a submonoid of the multiplicative monoid of $S_q$ that encodes a significant amount of arithmetic information about $S_q$. We then study several fundamental factorization properties of $S_q$: the bounded factorization (BF) and finite factorization (FF) properties, factoriality, and half-factoriality. We prove that, over the class of rational monogenic semidomains, the BF property is equivalent to the ascending chain condition on principal ideals, and we determine all values of the parameter $q$ for which $S_q$ satisfies the FF property. We also prove that $S_q$ is a unique factorization semidomain if and only if $q \in \nn \cup \nn^{-1}$. Finally, we prove that $S_q$ is a Krull semidomain if and only if it is root-closed, which happens precisely when $S_q$ is a unique factorization semidomain.
\end{abstract}

\bigskip
\maketitle

\bigskip
\section{Introduction}
\label{sec:intro}

In commutative ring theory, $\mathbb{Z}$ serves as the prototypical integral domain and as a model for studying more complex commutative rings, many of which are obtained by taking extensions of $\mathbb{Z}$, including polynomial rings, rings of integer-valued polynomials, Laurent polynomial rings, and, more generally, commutative semigroup rings. A similar role is played by $\nn_0$ in commutative semiring theory. Among commutative semirings of characteristic zero, the polynomial semiring $\nn_0[x]$ and, more generally, the semirings
\[
	\nn_0[\alpha] := \{p(\alpha) : p(x) \in \nn_0[x]\},
\]
where $\alpha \in \cc$, are perhaps the simplest extension semidomains of $\nn_0$ with nontrivial arithmetic structure and factorization behavior (of course, the semiring $\nn_0[\tau]$ is isomorphic to $\nn_0[x]$ for any transcendental $\tau \in \cc$). The semidomain $\nn_0[\alpha]$ is the smallest subsemiring of the field $\cc$ containing $\alpha$. More generally, consider an integral domain $D$ whose smallest subsemiring is denoted by $\mathbb{S}$; then $\mathbb{S}$ is either $\nn_0$ or the field $\ff_p$ of prime cardinality~$p$. For any $\alpha \in D$, one can consider the extension semiring $\mathbb{S}[\alpha]$ of $\mathbb{S}$ by $\alpha$. After the semidomains $\mathbb{S}$, which are trivial from the factorization viewpoint, the semidomains $\mathbb{S}[\alpha]$ are the simplest in terms of algebraic structure: we refer to $\mathbb{S}[\alpha]$ as the simple extension semidomain of $\mathbb{S}$ by~$\alpha$ or simply as a \emph{monogenic semidomain}.
\smallskip

Factorization in the setting of integral domains has been systematically studied by many authors for almost four decades (and even longer if one counts the study of factorization in the special case of Dedekind domains). As a result, a robust and refined theory of factorizations has been developed in the setting of integral domains, including the behavior of factorizations under relevant algebraic constructions such as polynomial and Laurent polynomial rings, monoid algebras, localizations, pullbacks, and directed unions. Given the fundamental nature of monogenic semidomains, any attempt to extend the factorization theory of integral domains to the more general setting of semidomains should include answers to the most essential questions involving the atomic structure of the semidomains $\mathbb{S}[\alpha]$.
\smallskip

The primary purpose of this paper is to initiate a systematic study of the atomicity, ideal theory, and factorization of monogenic semidomains $\nn_0[\alpha]$ for algebraic generators~$\alpha$, which account for the main nontrivial algebraic cases, since $\ff_p[\alpha]$ is a finite field when $\alpha \in \cc$ is algebraic. In this first investigation of monogenic semidomains $\nn_0[\alpha]$, we focus our attention on the case where the generator $\alpha$ is rational. Throughout this paper, we write 
\[
	S_q := \nn_0[q]
\]
for every $q \in \qq$. The results we establish here indicate that, even for rational generators $q$, understanding the atomic structure of the semidomain $S_q$ is not as simple as understanding the integral domain $\zz[q]$. Indeed, the fundamental question of whether $S_q$ is an atomic semidomain for every $q \in \qq$ remains open.
\smallskip

It is worth highlighting related recent work that has served as a source of motivation for our study of the multiplicative structure of monogenic semidomains of the form $\nn_0[\alpha]$. In 2019, Campanini and Facchini~\cite{CF19} provided a thorough study of the ideal-theoretic and factorization aspects of the polynomial semidomain $\nn_0[x]$. Their study covers, up to isomorphism, the class of monogenic semidomains $\mathbb{S}_0[\tau]$ for all transcendental $\tau \in \cc$: indeed, $\ff_p[\tau] \cong \ff_p[x]$ for every transcendental element $\tau \in \cc$, and $\ff_p[x]$ is a UFD and is therefore factorization-theoretically well understood.
\smallskip

In order to establish some of our primary findings, we rely on fundamental results on the additive structure of $S_q$ previously established by Chapman, Gotti, and Gotti~\cite{CGG20}. For rational generators $q$, the additive structure of the monogenic semidomains $S_q$ has also been studied by Jiang, Li, and Zhu~\cite{JLZ23}. The first systematic investigation of the additive structure of monogenic semidomains $\nn_0[\alpha]$ for algebraic generators $\alpha \in \cc$ was carried out by Correa-Morris and Gotti~\cite{CG22}, and this work motivated subsequent papers by Ajran et al.~\cite{ABLST23} and by Dani et al.~\cite{DDGx26} on the atomic structure and the factorization aspects of $\nn_0[\alpha]$, respectively.
\smallskip

Our primary purpose in this paper is to undertake a detailed study of the semirings $S_q$ parameterized by rationals $q$ through the lens of factorization theory. In this direction, we investigate the following properties, whose dependencies are illustrated in Figure~\ref{fig:AAZ with Krull property included}.
\smallskip
\begin{itemize}
	
	\item ACCP: every ascending chain of principal ideals in $S_q$ eventually stabilizes.
	\smallskip

	\item Bounded Factorization (BF) Property: existence of a length function on $S_q$.
	\smallskip

	\item Finite Factorization (FF) Property: every nonzero nonunit of $S_q$ has a nonempty and finite set of factorizations.
	\smallskip

	\item Krull Property: $S_q$ admits a divisor theory.
	\smallskip
	
	\item Half-Factoriality (HF) Property: any two factorizations of the same element of $S_q$ have the same length.
	\smallskip
	
	\item Unique Factorization (UF) Property: the statement of the Fundamental Theorem of Arithmetic holds.
\end{itemize}
\noindent The diagram in Figure~\ref{fig:AAZ with Krull property included} shows all the implications among the factorization properties weaker than the UF property that we have just listed. These are the main factorization properties we consider throughout this paper. A similar diagram was provided in the first systematic study of factorization theory in the setting of integral domains, carried out by Anderson, Anderson, and Zafrullah~\cite{AAZ90} in 1990. 
\begin{center}
	\begin{figure}[ht]
		\begin{tikzcd}[cramped]
			\textbf{ UF } \ \arrow[rr, Rightarrow, shift left=-0.3ex] \arrow[red, rr, Leftarrow, "/"{anchor=center,sloped}, shift left=1.2ex] \arrow[d, Rightarrow, shift right=1.3ex] \arrow[red, d, Leftarrow, "/"{anchor=center,sloped}, shift left=0.8ex] & & \ \textbf{ HF } \arrow[d, Rightarrow, shift right=0.6ex] \arrow[red, d, Leftarrow, "/"{anchor=center,sloped}, shift left=1.3ex] \\
			\textbf{ Krull } \ \arrow[r, Rightarrow, shift left=-0.3ex]	\arrow[red, r, Leftarrow, "/"{anchor=center,sloped}, shift left=1.3ex] & \ \textbf{ FF } \arrow[r, Rightarrow, shift left=-0.3ex]  \arrow[red, r, Leftarrow, "/"{anchor=center,sloped}, shift left=1.3ex] & \textbf{ BF }  \arrow[r, Rightarrow, shift left=-0.3ex] \arrow[red, r, Leftarrow, "/"{anchor=center,sloped}, shift left=1.3ex]  & \textbf{ACCP}.
		\end{tikzcd}
		\caption{The diagram shows the known implications among the atomic properties we consider in this paper in the general class of atomic semidomains. The diagram also emphasizes (with red marked arrows) that none of the shown implications is reversible in the class of semidomains.}
		\label{fig:AAZ with Krull property included}
	\end{figure}
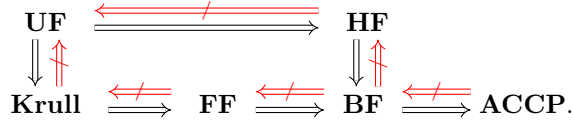
\end{center}

We conclude this section by describing the structure of this paper. In Section~\ref{sec:background}, we introduce the notation and terminology used throughout the paper, and we briefly cover the needed background on commutative monoids, commutative semirings, and factorization theory.
\smallskip

In Section~\ref{sec:algebraic aspects}, we begin our study of the multiplicative structure of the monogenic semidomains $S_q$ generated by a rational parameter~$q$. First, we compute the group of units and the Grothendieck domain of every rational monogenic semidomain. Also, for each rational generator $q$, we exress the semidomain $S_q$ as the union of an ascending chain $(M_{q,k})_{k \ge 0}$ whose terms are numerical monoids up to isomorphism, which allows us to treat $S_q$ as the directed union of finitely generated rank-$1$ monoids. 
\smallskip

In Section~\ref{sec:monoid T_q}, we introduce and study the monoid of technical fractions, which is a divisor-closed submonoid of $S_q^*$, denoted here by $T_q$. The monoid $T_q$ of technical fractions encodes relevant information about the arithmetic of the semidomain~$S_q$, and this will be evident in the section devoted to the study of factorizations. 
\smallskip

In Section~\ref{sec:factorization}, we investigate some of the most classical factorization properties in the setting of rational monogenic semidomains. We start the section studying the UF and the HF properties, proving that these two properties are equivalent when we restrict to the class of rational monogenic semidomains. As a byproduct of our work on factoriality, we were able to confirm that $\nn_0$ is the only rational monogenic semidomain that is a bi-UF, adding a small piece of evidence to a conjecture posed in 2021 by Baeth, Chapman, and the second author~\cite{BCG21}, which states that the only positive semiring whose additive and multiplicative monoids are both UFMs is $\nn_0$. This conjecture, known as the \emph{Bi-UF Positive Conjecture}, is still open. Then we determine the rational parameters $q$ such that $S_q$ satisfies the FF property, showing that $S_q$ has the FF property if and only if $q \ge 1$ or either $q$ is a unit fraction or its denominator is a power of a prime. We conclude the section proving that the BF property is equivalent to the ACCP, also arguing that these two equivalent properties hold precisely when $1$ is not a limit point of $T_q \setminus \{1\}$. 
\smallskip

In Section~\ref{sec:Krull property}, we first consider the root closure of the rational monogenic semidomains $S_q$, proving that $S_q$ is root-closed if and only if $q \in \nn \cup \nn^{-1}$, precisely the values of the parameters $q$ such that $S_q$ has the UF property. We conclude the section determining the semidomains $S_q$ having the Krull property, proving that $S_q$ has the Krull property if and only if its multiplicative monoid $S_q^*$ is root-closed, which happens precisely when $S_q$ is a factorial semidomain.
\smallskip

The main results we establish in this paper are summarized by the implication arrows shown in the following diagram:
\begin{center}
	\begin{figure}[ht]
		\begin{tikzcd}[cramped]
			[\textbf{ UF } \Longleftrightarrow \textbf{ HF } \Longleftrightarrow \textbf{ Krull }] \ \arrow[r, Rightarrow, shift left=-0.3ex] \arrow[red, r, Leftarrow, "/"{anchor=center,sloped}, shift left=1.2ex] \ & \  \textbf{ FF} \ \arrow[r, Rightarrow, shift left=-0.3ex] 
			 & \ [ \textbf{ BF } \ \Longleftrightarrow \ \textbf{ ACCP }].
		\end{tikzcd}
		\caption{The implications in the diagram above hold on the class of rational monogenic semidomains: they are the most relevant results established in this paper.}
		\label{fig:diagram with the results established in this paper}
	\end{figure}
\end{center}

\bigskip
\section{Background} \label{sec:background}

\medskip
\subsection{General Notation}

Throughout this paper, we let $\zz$, $\qq$, $\rr$, $\aaa$, and $\cc$ denote the set of integers, rational numbers, real numbers, algebraic complex numbers, and complex numbers, respectively. We let $\pp$, $\nn$, and $\nn_0$ denote the set of standard primes, positive integers, and nonnegative integers, respectively. For any $r,s \in \rr$, we set
\[
	\ldb r,s \rdb := \{n \in \zz : r \le n \le s\}.
\]
Observe that $\ldb r,s \rdb$ is empty for any $r,s \in \rr$ with $r > s$. For a subset $S$ of the real line, we set $S_{\ge r} := \{s \in S : s \ge r\}$ and $S_{> r} := \{s \in S : s > r\}$. For each $q \in \qq^\times$, there is a unique pair $(n,d) \in \zz \times \nn$ with $q = n/d$ and $\gcd(n,d) = 1$: we often denote $n$ and $d$ by $\mathsf{n}(q)$ and $\mathsf{d}(q)$, respectively. Finally, for each $p \in \pp$, we let $v_p \colon \qq^\times \to \zz$ denote the standard $p$-adic valuation map: $v_p(n) := \max\{m \in \nn_0 : p^m \mid n \}$ for every $n \in \zz \setminus \{0\}$ while $v_p(q) = v_p(\mathsf{n}(q)) - v_p(\mathsf{d}(q))$ for all $q \in \qq^\times$.

\medskip
\subsection{Commutative Monoids}

Every semigroup we mention or deal with throughout this paper is assumed to be commutative. Although a monoid is usually defined as a semigroup with an identity element, throughout this paper we tacitly assume that all monoids we refer to or deal with are cancellative (and also commutative due to the same assumption on semigroups). A \emph{monoid homomorphism} is a semigroup homomorphism that preserves the identity elements.
\smallskip

Let $M$ be a monoid. Following notation from commutative ring theory, we let $M^\times$ denote the abelian group consisting of all units (i.e., invertible elements) of $M$. If~$M^\times$ is the trivial group then we say that~$M$ is \emph{reduced}. The set $M_{\text{red}} := \{bM^\times : b \in M\}$ consisting of all cosets determined by $M^\times$ is also a reduced monoid, which is trivial if and only if $M$ is a group. The \emph{Grothendieck group} of~$M$ is the unique abelian group $\gp(M)$ up to isomorphism satisfying that any abelian group containing an isomorphic copy of the monoid~$M$ also contains an isomorphic copy of the group~$\gp(M)$.  If $S$ is a subset of $M$ then we let $\langle S \rangle$ denote the smallest submonoid of~$M$ containing~$S$. The monoid~$M$ is called \emph{finitely generated} if $M = \langle S \rangle$ for a finite subset $S$ of $M$.
\smallskip

A subset $I$ of $M$ is called an \emph{ideal} of $M$ if $MI \subseteq I$. An ideal $I$ is called \emph{principal} if it can be written as $I = b M$ for some $b \in M$. An ascending chain of ideals $(I_n)_{n \ge 1}$ \emph{starts} at $bM$ if $I_1 = bM$. An element $b \in M$ satisfies the \emph{ascending chain of principal ideals} (ACCP) if for every sequence of principal ideals $(I_n)_{n \ge 1}$ starting at $bM$ there exists $m \in \nn$ such that $I_n = I_m$ for every $n \ge m$. If every element of $M$ satisfies the ACCP then we say that $M$ satisfies the \emph{ACCP}. For nonempty subsets $S$ and $T$ of $\gp(M)$, we set
\[
	(T:S) := \{g \in \gp(M) : gS \subseteq T \}.
\]
Then we set $S_v := (M:(M:S))$ and we call $S_v$ the \emph{divisorial closure} of $S$. If $S_v = S$ then we say that $S$ is \emph{divisorially closed}. An ideal of $M$ is said to be \emph{divisorial} or a $v$-\emph{ideal} provided that it is divisorially closed, and $M$ is said to be $v$-Noetherian provided that every ascending chain of principal ideals of $M$ eventually stabilizes. Because $(sM)_v = (M:(M:sM)) = sM$, it follows that every principal ideal is divisorial and so every $v$-Noetherian monoid satisfies the ACCP.
\smallskip

There are two closure constructions of $M$ inside $\gp(M)$ that are relevant in the scope of this paper: the root closure and the complete integral closure. The \emph{root closure} of $M$, denoted by $\bar{M}$, is the following subset of $\mathcal{G}(M)$:
\[
	\bar{M} := \big\{ g \in \mathcal{G}(M) : g^n \in M \text{ for some } n \in \nn \big\}.
\]
One can readily verify that $\bar{M}$ is a monoid such that $M \subseteq \bar{M} \subseteq \mathcal{G}(M)$. We say that $M$ is \emph{root-closed} provided that $\bar{M} = M$. An element $g \in \gp(M)$ is called \emph{almost integral} provided that there exists $c \in M$ such that $cg^n \in M$ for every $n \in \nn$. It turns out that the set
\[
	\widehat{M} := \{g \in \gp(M) : \exists \, c \in M \ \text{ such that } \ cg^n \in M \ \text{ for every } \ n \in \nn \}
\]
consisting of all almost integral elements of $M$ is a monoid called the \emph{complete integral closure} of~$M$. It turns out that $\widehat{M}$ is a submonoid of $\gp(M)$. We say that $M$ is \emph{completely integrally closed} if $\widehat{M} = M$\footnote{The complete integral closure of a monoid $M$ is not, in general, a closure operation in the sense that it may be that $\widehat{\widehat{M}} \neq \widehat{M}$.}. It is not hard to verify that the following chain of inclusions holds:
\[
	M \subseteq \bar{M} \subseteq \widehat{M} \subseteq \gp(M), 
\]
and so every completely integrally closed monoid is root-closed.
\smallskip

We say that $M$ is a \emph{Krull monoid} if $M$ is both $v$-Noetherian and completely integrally closed. Therefore every Krull monoid is root-closed and satisfies the ACCP.

\medskip
\subsection{Atomicity and Factorizations}

An element $a \in M \! \setminus M^\times$ is called an \emph{atom} if whenever $a = uv$ for some $u,v \in M$ then either $u \in M^\times$ or $v \in M^\times$.  The set consisting of all the atoms of $M$ is denoted by $\mathcal{A}(M)$. An element of~$M$ is called \emph{atomic} if it is a unit or it factors into finitely many atoms in~$M$ (allowing repetitions). Following Cohn~\cite{pC68}, we say that~$M$ is \emph{atomic} if every element of~$M$ is atomic. It is not hard to verify that every monoid that satisfies the ACCP is atomic and so every Krull monoid is atomic.
\smallskip

We assume, for the rest of this section, that $M$ is an atomic monoid. Let $\mathsf{Z}(M)$ denote the free commutative monoid on the set $\mathcal{A}(M_{\text{red}})$, and let $\pi \colon \mathsf{Z}(M) \to M_\text{red}$ denote the only monoid homomorphism that fixes every element of the set $\mathcal{A}(M_{\text{red}})$. The elements of $\mathsf{Z}(M)$ are called \emph{factorizations}. For each $m \in M$, we set
\[
\mathsf{Z}(m) := \pi^{-1} (m M^\times).
\]
The monoid~$M$ is called a \emph{unique factorization monoid} (UFM) if $|\mathsf{Z}(m)| = 1$ for all $m \in M$. It is well known and not difficult to verify that every UFM is a Krull monoid. Therefore every UFM is root-closed. The monoid $M$ is called a \emph{finite factorization monoid} (FFM) if $|\mathsf{Z}(m)| < \infty$ for all $m \in M$. It follows from the definition that every UFM is an FFM. We say that a monoid has the UF (resp., FF) property provided that it is a UFM (resp., an FFM). It is well known that the class of Krull monoids lies between the class of UFMs and the class of FFMs:
\begin{center}
	\begin{figure}[ht]
		\begin{tikzcd}[cramped]
			\textbf{ UF } \arrow[r, Rightarrow, shift left=-0.3ex]  \arrow[red, r, Leftarrow, "/"{anchor=center,sloped}, shift left=1.3ex] & \textbf{ Krull }  \arrow[r, Rightarrow, shift left=-0.3ex] \arrow[red, r, Leftarrow, "/"{anchor=center,sloped}, shift left=1.3ex]  & \textbf{ FF}.
		\end{tikzcd}
		\caption{Every monoid with the UF property is a Krull monoid, while every Krull monoid has the FF property.}
		\label{fig:Krull monoids between UFMs and FFMs}
	\end{figure}
\end{center}
Given a factorization $z \in \mathsf{Z}(M)$, we refer to the \emph{length} of $z$ as the number of atoms of $M_{\text{red}}$ that appear in $z$ (counting repetitions), and we let $|z|$ denote the length of $z$: if $z = a_1 \cdots a_\ell$ for some $a_1, \dots, a_\ell \in \mathcal{A}(M_{\text{red}})$ then $\ell$ is the length of $z$. For each $m \in M$, the \emph{set of lengths} of $m$ is defined as follows:
\[
\mathsf{L}(m) := \{ |z| : z \in \mathsf{Z}(m) \}.
\]
The monoid $M$ is called a \emph{bounded factorization monoid} (BFM) if $|\mathsf{L}(m)| < \infty$ for all $m \in M$, while $M$ is called a \emph{half-factorial monoid} (HFM) if $|\mathsf{L}(m)| = 1$ for all $m \in M$. Directly from the corresponding definitions, one can deduce that a monoid is a BFM provided that it is either an FFM or an HFM. In addition, it is not difficult to argue that every BFM satisfies the ACCP. We say that a monoid has the BF (resp., HF) property provided that it is a BFM (resp., an HFM).

\medskip
\subsection{Semidomains}

Let $S$ be a nonempty set, and let $+$ and $\cdot$ be two binary operations on~$S$ called \emph{addition} and \emph{multiplication}, respectively. We say that the triple $(S,+, \cdot)$ is a \emph{semiring} provided that $(S,+)$ is an additive monoid, $(S, \cdot)$ is a multiplicative commutative semigroup with identity element, and the multiplicative operation~$\cdot$ distributes over the additive operation~$+$.\footnote{As every semiring we deal with in the scope of this paper is commutative, we took the freedom to refer to any commutative semiring simply as a `semiring'.} The identity elements of the monoids $(S,+)$ and $(S, \cdot)$ are denoted by~$0$ and~$1$, respectively. We often write a semiring $(S,+,\cdot)$ simply as $S$. For a semiring $S$, we let $(S,+)$ denote the additive monoid of~$S$.
\smallskip

Let $S$ and $T$ be semirings. A map $S \to T$ is called a \emph{semiring homomorphism} if it is a monoid homomorphism between the additive monoids of $S$ and $T$ and an identity-preserving semigroup homomorphism between the multiplicative monoids of $S$ and $T$. A subset $S'$ of a semiring $S$ is called a \emph{subsemiring} of $S$ provided that $S'$ is closed under both operations of~$S$ and contains both the additive and the multiplicative identities of~$S$.
\smallskip

A \emph{semidomain} is a subsemiring of an integral domain. Let us assume, for the rest of this section, that $S$ is a semidomain. Then the subset
\[
	S^* := S \setminus \{0\}
\]
of $S$ is closed under multiplication, and so $S^*$ is a multiplicative submonoid of~$S$. When $S$ is a semidomain, we call~$S^*$ the \emph{multiplicative monoid} of~$S$. In the same way that every monoid can be minimally embedded into its Grothendieck group, the semidomain $S$ can be minimally embedded into an integral domain $\mathcal{D}(S)$ in such a way that every integral domain containing an isomorphic copy of $S$ also contains an isomorphic copy of $\mathcal{D}(S)$. The integral domain $\mathcal{D}(S)$ is unique up to ring isomorphism, and we call it the \emph{Grothendieck domain} of $S$. Observe that the field of fractions of $\mathcal{D}(S)$ is the smallest field containing $S$ and we call it the \emph{field of fractions} of $S$.
\smallskip

We say that the semidomain $S$ is \emph{reduced} (\emph{atomic}, a \emph{Krull semidomain}, a \emph{bounded factorization semidomain} (BFS), a \emph{finite factorization semidomain} (FFS), a \emph{half-factorial semidomain} (HFS), a \emph{unique factorization semidomain} (UFS)) if the multiplicative monoid $S^*$ is reduced (resp., atomic, a Krull monoid, a BFM, an FFM, an HFM, a UFM).
\smallskip

Let $F$ be a field containing $S$ as a subsemiring. The subsemiring of $F[x]$ consisting of all the polynomials with coefficients in $S$ is clearly a semidomain, which we denote by $S[x]$. For any $\alpha \in F$, the subsemiring
\[
	S[\alpha] := \big\{ p(\alpha) : p(x) \in S[x] \big\}
\]
of $F$ is the intersection of all subsemirings of~$F$ containing both~$S$ and~$\alpha$, and we refer to $S[\alpha]$ as the \emph{extension semidomain} of $S$ by $\alpha$. 
\smallskip

As the semidomain $S$ can be embedded into an integral domain, the intersection of all subsemirings of $S$, which we denote by $\mathbb{S}$, is either $\nn_0$ or $\ff_p$ for some $p \in \pp$. We call $\mathbb{S}$ the \emph{prime subsemiring} of~$S$. The central objects of this paper are the semidomains which are the simple extension semirings of their prime subsemirings.

\begin{defn}
	Let $S$ be a semidomain with prime subsemiring~$\mathbb{S}$. We say that $S$ is a \emph{monogenic semidomain} if there exists $\alpha \in S$ such that $S = \mathbb{S}[\alpha]$, in which case, we also say that $S$ is the \emph{monogenic semidomain generated by} $\alpha$ (\emph{over} $\mathbb{S}$).
\end{defn}

Observe that when $\mathbb{S} = \ff_p$ for some $p \in \pp$, the monogenic semidomain $\mathbb{S}[\alpha]$ is actually a subring of $F$ (because $-1 \in \ff_p[\alpha]$). Thus, when $\alpha \in F$ is the root of a polynomial in $\ff_p[x]$ (the algebraic case), the monogenic semidomain $\ff_p[\alpha]$ is actually the field $\ff_p(\alpha)$ and, otherwise (the transcendental case), the monogenic semidomain $\ff_p[\alpha]$ of $\ff_p$ is isomorphic to the polynomial ring $\ff_p[x]$, which is a UFD and so trivial from the factorization viewpoint. Thus, in the scope of this paper, we restrict our attention to the case $\mathbb{S} = \nn_0$.
\smallskip

We aim for this paper to be the first systematic study of the multiplicative structure of monogenic semidomains. Thus, we have restricted ourselves to focus only on the monogenic extensions of the prime subsemiring $\nn_0$ by rational generators. This allows us to fix our ambient field $F$ to be $\qq$. For the sake of simplicity, we introduce some notation we shall be using throughout the paper.
\smallskip

\noindent \textbf{Notation.} For $q \in \qq$, we let $S_q$ denote the monogenic semidomain generated by $q$ over $\nn_0$, and we let $M_q$ denote the additive monoid of the semidomain $S_q$.
\smallskip

One can readily verify that if a submonoid of the additive group of $\qq$ contains both a positive and a negative element then it must be a subgroup. Thus, when the generator $q \in \qq$ is negative, $S_q = \zz[q]$, which is the localization of $\zz$ at the multiplicative set generated by the rational primes dividing $\mathsf{d}(q)$, which is a PID. As a consequence, when $q$ is a negative rational, $S_q$ is trivial from the factorization perspective. Hence in the scope of this paper we restrict our attention to monogenic semidomains $S_q$ whose generating parameter $q$ is a positive rational, which we call rational monogenic semidomains.
\smallskip

The atomic structure and the arithmetic of factorizations of the additive monoid of $S_q$ have been considered in the past few years (for instance, see~\cite{CGG20,CG22}). We will use the following known results in coming sections.

\begin{thm} \emph{(}\cite[Theorem~6.2]{GG18}, \cite[Theorem~4.11]{CG22}\emph{)}
	For $q \in \qq_{> 0}$, let $M_q$ denote the additive monoid of the rational monogenic semidomain $S_q$. Then $M_q$ is atomic if and only if $q^{-1} \notin \nn_{\ge 2}$, in which case the following statements hold.
	\begin{itemize}
		\item If $q \in \nn$ then $\mathcal{A}(M_q) = \{1\}$ (as $M_q = \nn_0$ in this case).
		\smallskip
		
		\item If $q \notin \nn$ then $\mathcal{A}(M_q) = \{q^n : n \in \nn_0 \}$.
	\end{itemize}
	Moreover, when $M_q$ is atomic, the following conditions are equivalent.
	\begin{enumerate}
		\item[(a)] $M_q$ satisfies the ACCP.
		\smallskip
		
		\item[(b)] $M_q$ is a BFM.
		\smallskip
		
		\item[(c)] $M_q$ is an FFM.
		\smallskip
		
		\item[(d)] $q \ge 1$.
	\end{enumerate}
\end{thm}

\bigskip
\section{Algebraic Aspects} \label{sec:algebraic aspects}

In this section, we discuss various algebraic aspects of the rational monogenic semidomains $S_q$. We start by introducing, for each $q \in \qq_{>0}$, a sequence of numerical monoids associated to $S_q$. Then we determine the divisibility group of $S_q$ for every possible value of the generator~$q$.

\medskip
\subsection{Group of Units and Grothendieck Domain}

It turns out that the semidomain $S_q$ is reduced when $q^{-1} \notin \nn_{\ge 2}$ and a finite-rank free abelian group if $q^{-1} \in \nn_{\ge 2}$.

\begin{prop}\label{prop:units of N_0[q]}
	For $q \in \qq_{>0}$, let $S_q$ be the rational monogenic semidomain generated by~$q$. Then the following statements hold.
	\begin{itemize}
		\item $S_q^\times = \gp(p \in \pp : v_p(q) \neq 0)$ if $q^{-1} \in \nn_{\ge 2}$.
		\smallskip
		
		\item $S_q^\times$ is the trivial group if $q^{-1} \notin \nn_{\ge 2}$.
	\end{itemize}
\end{prop}

\begin{proof}
	We split the proof into the following two cases.
	\smallskip
	
	\textsc{Case 1:} $q^{-1} \in \nn_{\ge 2}$.  In this case, we can write $q = 1/d$ for some $d \in \nn$ with $d \ge 2$. Let $p_1, \dots, p_\ell$ be the primes dividing~$d$. Let $M$ be the multiplicative submonoid of $S_q^*$ generated by the set $\{p_1, \dots, p_\ell\}$, and let $\gp(M)$ be the Grothendieck group of $M$ inside $\gp(S_q^*)$. Assume, without loss of generality, that $p_1 < \dots < p_\ell$. For each $i \in \ldb 1,\ell \rdb$, we can write $1/p_i$ as the sum of $d/p_i$ copies of $1/d \in S_q$, whence $1/p_i \in S_q^\times$. As $p_i$ is a unit in $S_q$ for every $i \in \ldb 1,\ell \rdb$, the following inclusion holds:
	\[
		\gp(M) = \big\{ p_1^{e_1} \dots p_\ell^{e_\ell} : (e_1, \dots, e_\ell) \in \zz^{\ell} \big\} \subseteq S_q^\times.
	\]
	To argue the reverse inclusion, take $u \in S_q^\times$. Then every prime dividing either $\mathsf{n}(u)$ or $\mathsf{d}(u)$ also divides~$d$, and so $\mathsf{n}(u), \mathsf{d}(u) \in M$. Therefore $u = \mathsf{n}(u)\mathsf{d}(u)^{-1} \in \gp(M)$. Hence the inclusion $S_q^\times \subseteq \gp(M)$ also holds.
	\smallskip
	
	\textsc{Case 2:} $q^{-1} \notin \nn_{\ge 2}$. If $q=1$ then $S_q = \nn_0$ and it is clear that $\nn_0^\times$ is the trivial group. Now observe that if $q \ge 1$ then $\min S_q^* \ge 1$: in this case, the only nonzero element of $S_q$ having a multiplicative inverse is $1$, and so $S_q$ is reduced. Therefore we will assume that $q < 1$. Since $\{q^n : n \in \nn_0\}$ is a generating set of the additive monoid of $S_q$, any $s \in S_q$ with $s<1$ must belong to the Puiseux monoid $\langle q^n : n \in \nn \rangle$ and so $v_p(s) \ge 1$ for any prime divisor $p$ of~$\mathsf{n}(q)$. Thus, $v_p(1/s) = -v_p(s) \le -1$ for each prime divisor $p$ of $\mathsf{n}(q)$, which implies that $1/s \notin S_q$. Therefore $S_q$ does not contain any unit less than~$1$, which implies that $S_q$ cannot contain any unit greater than $1$ either. Hence $S_q$ is reduced.
\end{proof}
	
Recall that, for each $q \in \qq_{>0}$, we let $M_q$ denote the additive monoid of the monogenic semidomain~$S_q$. Note that the Grothendieck group of the additive monoid $M_q$ is a subgroup of the additive abelian group $\zz[q]$, so the fact that every abelian group containing $\gp(M_q)$ must contain $-1$ (and so $\zz$) ensures that $\zz[q]$ is the Grothendieck group of $M_q$.

\begin{prop}
	Let $S_q$ be the monogenic semidomain generated by $q \in \qq_{>0}$. Then the following statements hold.
	\begin{enumerate}
		\item $\mathcal{G}(S_q^*) = \qq^\times$.
		\smallskip
		
		\item $\mathcal{D}(\nn_0[q]) = \zz[q]$.
	\end{enumerate}
\end{prop}

\begin{proof}
	(1) This part follows immediately from the fact that $\nn_0[q]^*$ is an overmonoid of the multiplicative monoid $\nn$, whose Grothendieck group is $\qq^\times$.
	\smallskip
	
	(2) Since $S_q$ is a subsemiring of the integral domain $\zz[q]$, we can assume that the inclusions $S_q \subseteq \mathcal{D}(S_q) \subseteq \zz[q]$ hold. As $\mathcal{D}(S_q)$ is an abelian group under addition, it must contain~$-1$ and so $\zz \subseteq \mathcal{D}(S_q)$. This, along with the fact that $q^n \in \mathcal{D}(S_q)$ for every $n \in \nn$, ensures that $\zz[q] \subseteq \mathcal{D}(S_q)$, whence $\zz[q]$ must actually be the Grothendieck domain of $S_q$.
\end{proof}

\medskip
\subsection{Associated Numerical Monoids}

As in the previous section, we fix the generator $q \in \qq_{>0}$ of~$S_q$. Write $q = n/d$, with $n,d \in \nn$ such that $\gcd(n,d) = 1$. For each $k \in \nn_0$, we call the numerical monoid
\[
	N_{q,k} := \big\langle n^j d^{k-j} : j \in \ldb 0,k \rdb \big\rangle,
\]
the $k$-\emph{th numerical monoid associated} to the semidomain $S_q$. Observe that $N_{q,0} = \nn_0$. One can readily verify that $\{n^j d^{k-j} : j \in \ldb 0,k \rdb \}$ is a minimal generating set for $N_{q,k}$, and so
\[
	\mathcal{A}(N_{q,k}) = \big\{ n^jd^{k-j} : j \in \ldb 0,k \rdb \big\}.
\]
For each $k \in \nn$, we let $\mathsf{f}_{q,k}$ denote the Frobenius number of the numerical monoid $N_{q,k}$. The formula
\[
	\mathsf{f}_{q,k} = \frac{d^{k+2}-n^{k+2}}{d-n}-\frac{d^{k+1}-n^{k+1}}{d-n} - (d^{k+1}+n^{k+1}) = \frac{(n-1)d^{k+1} - (d-1)n^{k+1}}{d-n}
\]
was established by Ong and Ponomarenko in~\cite{OP08} (see also \cite{aT08}). Throughout this paper, often we will refer to the terms of the sequence $(\mathsf{f}_{q,k}/d^k)_{k \ge 1}$ and its limit (when it exists), so the following notation will be helpful: for each $k \in \nn$,
\[
	B_{q,k} := \frac{\mathsf{f}_{q,k}}{d^k}  = \frac{d(n-1) - n q^k(d-1)}{d-n}. 
\]
When $q \in (0,1)$, the limit of the sequence $(B_{q,k})_{k \ge 1}$ exists and will also be helpful, so in this case we set:
\begin{equation} \label{eq:Bq infinite}
	B_{q,\infty} := \lim_{k \to \infty} B_{q,k} = \lim_{k \to \infty} \frac{d(n-1) - n q^k(d-1)}{d-n} = \frac{d(n-1)}{d-n}.
\end{equation}
In what follows, it is often more convenient to use the sequence of finitely generated Puiseux monoids $\big( M_{q,k} \big)_{k \ge 0}$ whose $k$-th term is isomorphic to $N_{q,k}$:
\[
	M_{q,k} := \frac1{d^k} N_{q,k} = \big\langle q^j : j \in \ldb 0,k \rdb \big\rangle
\]
for every $k \in \nn_0$. Observe that if $q \notin \nn_0$ then $q^{k+1} \notin M_{q,k}$ for any $k \in \nn_0$, and so the sequence $\big( M_{q,k} \big)_{k \ge 0}$ is a strictly ascending chain of submonoids of the additive monoid of $S_q$ with $M_{q,0} = \nn_0$ and
\[
	\bigcup_{k \in \nn_0} M_{q,k} = S_q.
\]

The following lemma about membership in the rational monogenic semidomain $S_q$ was first established as part of CrowdMath 2024 in the study of the Goldbach statement for rational monogenic semidomains, but we include this CrowdMath lemma here for the sake of completeness.

\begin{lem} \cite{CM24} \label{lem:Frobenius bound}
	For $q \in \qq_{>0} \setminus (\nn \cup \nn^{-1}$, let $S_q$ be the rational monogenic semidomain generated by $q$. Then the following statements hold.
	\begin{enumerate}
		\item For each $k \in \nn$, if $c/\mathsf{d}(q)^k > B_{q,k}$ for some $c \in \nn$ then $c/\mathsf{d}(q)^k \in S_q$.
		\smallskip
		
		\item If $q \in (0,1)$ then, for each $k \in \nn$, if $c/\mathsf{d}(q)^k \ge B_{q,\infty}$ for some $c \in \nn$ then $c/\mathsf{d}(q)^k \in S_q$.
	\end{enumerate} 
\end{lem}
	
\begin{proof}
	(1) Write $q = n/d$ in lowest terms, and fix an index $k \in \nn$. Then take $c \in \nn$ such that $c/d^k > B_{q,k}$. As $c > d^kB_{q,k} = \mathsf{f}_{q,k}$, we see that $c \in N_{q,k}$ and, as a consequence, we can take coefficients $c_0, \dots, c_k \in \nn_0$ such that
	\[
		c =\sum_{j=0}^k c_j n^j d^{k-j}.
	\]
	Finally, we can divide the previous equality by $d^k$ to obtain the desired inclusion as follows:
	\[
		\frac{c}{d^k} = \sum_{j=0}^k c_j \left(\frac{n}{d}\right)^j = \sum_{j=0}^k c_jq^j \in S_q.
	\]
	\smallskip
	
	(2) Now assume that $q \in (0,1)$ and, as in the previous part, write $q = n/d$ in lowest terms. If $n=1$ then $q=1/d$ and so  $c/d^k = cq^k \in S_q$ for all $c,k \in \nn$. Therefore we can assume that $n\ge 2$, which implies that $\min\{n,d\} \ge 2$. Now fix $k \in \nn$ arbitrarily, and then take $c \in \nn$ such that $c/d^k \ge B_{q,\infty}$. As $d \ge n \ge 2$,
	\[
		B_{q,k} = \frac{d(n-1)}{d-n} - \frac{nq^k(d-1)}{d-n} < \frac{d(n-1)}{d-n} = B_{q,\infty},
	\]
	so $c/d^k \ge B_{q,\infty} > B_{q,k}$. Then it follows from part~(1) that $c/d^k \in S_q$, which completes the proof.
\end{proof}

\begin{cor}\label{cor:Frobenius bound for q<1}
	For $q \in \qq \cap (0,1)$, let $S_q$ be the rational monogenic semidomain generated by $q$. For each $k \in \nn$, if $c/\mathsf{d}(q)^k \ge B_{q,\infty}$ for some $c \in \nn$ then $c/\mathsf{d}(q)^k \in S_q$.
\end{cor}

\begin{proof}
	This is precisely part~(2) of Lemma~\ref{lem:Frobenius bound}.
\end{proof}

\medskip
\subsection{Polynomial Representation}

Next, we discuss polynomial representations of elements of the rational monogenic semidomain $S_q$, where $q \in \qq_{>0}$. Given a polynomial $f(x) \in \nn_0[x]$, we let $\deg f$ and $\text{ord} \, f$ denote the degree and order of $f$, respectively, and we let $[x^n]f(x)$ denote the coefficient of $x^n$ in $f(x)$.

\begin{defn}
	For $q \in \qq_{>0}$, let $S_q$ be the rational monogenic semidomain generated by $q$. Given $s \in S_q$, we say that $f(x) \in \nn_0[x]$ is a \emph{polynomial representation} of $s$ in $S_q$ if $f(q) = s$, and we say that a polynomial representation $f(x)$ of $s$ is \emph{optimal} if $f(1) \le g(1)$ for every polynomial representation $g(x)$ of~$s$.
\end{defn}

It follows immediately from the definition of $S_q$ that every element of $S_q$ has a polynomial representation. Here is an illustrative example.

\begin{exam}
	\hfill
	\begin{itemize}
		\item First, fix $q := 2/5$ and consider the rational monogenic semidomain $S_q$. As $5q^2=2q$, the polynomials $5x^2+1$ and $2x+1$ represent the same element of $S_q$. Thus, the replacement $5x^2 \mapsto 2x$ transforms the polynomial representation $5x^2+1$ into the optimal polynomial representation $2x+1$ of the element $9/5$. This illustrates the reduction process in the case $q<1$: coefficients of positive powers are reduced modulo $\den(q)$ by moving mass to lower degree.
		\smallskip

		\item Now fix $q := 5/3$ and consider the rational monogenic semidomain $S_q$. As $5q=3q^2$, the polynomials $5x+1$ and $3x^2+1$ represent the same element of $S_q$. Thus, the replacement $5x \mapsto 3x^2$ transforms the polynomial representation $5x+1$ into the optimal polynomial representation $3x^2+1$ of the element $28/3$. This illustrates the reduction process in the case $q>1$: coefficients are reduced modulo $\num(q)$ by moving mass to higher degree. \hfill $\blacksquare$
	\end{itemize}
\end{exam}

We proceed to characterize optimal polynomial representations and to prove they are unique.

\begin{prop} \label{prop:optimal polynomial representations}
	For $q \in \qq_{>0} \setminus \{1\}$, let $S_q$ be the rational monogenic semidomain generated by~$q$. Then the following statements hold.
	\begin{itemize}
		\item[(1)] If $q<1$ then, for each $s \in S_q$, there is a unique polynomial representation $f(x)$ of $s$ such that $[x^i]f(x) < \mathsf{d}(q)$ for every $i \in \nn$. Moreover, $f(x)$ is the unique optimal polynomial representation of~$s$ in $S_q$.
		\smallskip
		
		\item[(2)] If $q>1$ then, for each $s \in S_q$, there is a unique polynomial representation $f(x)$ of $s$ such that $[x^i]f(x) < \mathsf{n}(q)$ for every $i \in \nn_0$. Moreover, $f(x)$ is the unique optimal polynomial representation of~$s$ in $S_q$.
	\end{itemize}
	
	In particular, every $s \in S_q$ has a unique optimal polynomial representation in $S_q$.
\end{prop}

\begin{proof}
	Write $q = n/d$ for some $n,d \in \nn$ with $\gcd(n,d)=1$.
	\smallskip
	
	(1) Assume first that $q < 1$, and so $d > n$. To prove the existence of the desired polynomial representation, fix $s \in S_q$. Let $h(x) \in \nn_0[x]$ be any polynomial representation of $s$. As long as there exists $j \in \nn$ such that $[x^j]h(x) \ge d$, apply the following replacement:
	\[
		dx^j \mapsto nx^{j-1}.
	\]
	This replacement preserves the value of $h(x)$ at $q$, since $dq^j = nq^{j-1}$ while decreases the value of $h(x)$ at $1$ by $d-n>0$. Hence the process must terminate and, when it does, we obtain a polynomial representation $f(x)$ of $s$ satisfying $[x^i]f(x)<d$ for every $i \in \nn$.
	
	We now prove uniqueness. Suppose that $f(x),g(x) \in \nn_0[x]$ are two polynomial representations of the same element $s \in S_q$ such that $\max\{[x^i]f(x),[x^i]g(x)\} < d$ for every $i \in \nn$. Assume, toward a contradiction, that $f(x) \ne g(x)$, and set
	\[
		m:=\max\{i \in \nn_0 : [x^i]f(x) \ne [x^i]g(x)\}.
	\]
	We now split the argument into two cases and derive a contradiction in each of them.
	\smallskip
	
	\noindent \textsc{Case 1:} $m=0$. In this case, the equality $f(q)=g(q)$ forces $[x^0]f(x)=[x^0]g(x)$, which is the desired contradiction.
	\smallskip
	
	\noindent \textsc{Case 2:} $m \ge 1$. Since $f(q)=g(q)$, we see that
	\[
		0=\sum_{i=0}^m \big([x^i]g(x)-[x^i]f(x)\big)n^id^{m-i}.
	\]
	Reducing this equality modulo $d$, we obtain
	\[
		n^m\big([x^m]g(x)-[x^m]f(x)\big) \equiv 0 \pmod d.
	\]
	As $\gcd(n,d)=1$, it follows that $d$ divides $[x^m]g(x)-[x^m]f(x)$. Since this difference has absolute value less than $d$, it must be zero, contradicting the definition of $m$. Therefore $f(x)=g(x)$.
	\smallskip
	
	Finally, observe that the polynomial representation constructed above is optimal because every polynomial representation of $s$ can be transformed into $f(x)$ by replacements that decrease the value at $1$. Conversely, if $h(x)$ is an optimal polynomial representation of $s$ then no coefficient $[x^j]h(x)$ with $j \in \nn$ can be at least $d$, since replacing $dx^j$ by $nx^{j-1}$ would produce a polynomial representation of $s$ with smaller value at $1$. Thus, the inequality $[x^i]h(x) < d$ holds for every $i \in \nn$, and the uniqueness just proved gives $h(x)=f(x)$. Hence $f(x)$ is the unique optimal polynomial representation of~$s$.
	\smallskip
	
	(2) Now assume that $q > 1$, and so $n>d$. To prove existence, fix $s \in S_q$. Start from any polynomial representation $h(x) \in \nn_0[x]$ of $s$ and, as long as there exists $j \in \nn_0$ such that $[x^j]h(x) \ge n$, apply the following replacement:
	\[
		nx^j \mapsto dx^{j+1}.
	\]
	As in the previous part, observe that this replacement preserves the value of $h(x)$ at $q$ because $nq^j=dq^{j+1}$ and decreases the value of $h(x)$ at $1$ by $n-d>0$. Hence the process terminates, yielding a polynomial representation $f(x)$ of $s$ satisfying $[x^i]f(x)<n$ for every $i \in \nn_0$.
	
	To prove uniqueness, assume toward a contradiction that $f(x),g(x) \in \nn_0[x]$ are two distinct polynomial representations of the same element $s \in S_q$ such that $\max\{[x^i]f(x),[x^i]g(x)\} < n$ for every $i \in \nn_0$. Then set
	\[
		m:=\min\{i \in \nn_0 : [x^i]f(x) \ne [x^i]g(x)\} \quad \text{ and } \quad r:=\max\{\deg f,\deg g\}.
	\]
	Observe that dividing both sides of the equality $f(q)=g(q)$ by $q^m$ and then multiplying them by the factor $d^{r-m}$ yields
	\[
		0=\sum_{i=m}^r \big([x^i]g(x)-[x^i]f(x)\big)n^{i-m}d^{r-i},
	\]
	which we can reduce modulo $n$ to obtain $d^{r-m}\big([x^m]g(x)-[x^m]f(x)\big) \equiv 0 \pmod n$. As $\gcd(n,d)=1$, it follows that $n$ divides $[x^m]g(x)-[x^m]f(x)$. Since this difference has absolute value less than $n$, we obtain that $[x^m]f(x) = [x^m]g(x)$. However, this contradicts the definition of~$m$. Hence $f(x) = g(x)$.
	
	The proof of optimality is analogous to that in part~(1). The constructed representation $f(x)$ is optimal because every polynomial representation of $s$ can be transformed into $f(x)$ by replacements that decrease the value at $1$. Conversely, if an optimal polynomial representation $h(x)$ of $s \in S_q$ had a coefficient $[x^j]h(x) \ge n$ for some $j \in \nn_0$ then replacing $nx^j$ by $dx^{j+1}$ would produce another polynomial representation of the same element with smaller value at $1$, a contradiction. Thus, $[x^i]h(x)<n$ for every $i \in \nn_0$, and so $h(x)=f(x)$ by the uniqueness already proved. Hence $f(x)$ is the unique optimal polynomial representation of~$s$.
\end{proof}

Here is a consequence of Proposition~\ref{prop:optimal polynomial representations} that we will use later on.

\begin{corollary} \label{cor:unique polynomial representation if all coefficients are zero or one}
	For $q \in \qq_{>0} \setminus (\nn \cup \nn^{-1})$, let $S_q$ be the rational monogenic semidomain generated by~$q$. If $f(x) \in \nn_0[x]$ has all coefficients equal to $0$ or $1$ then $f(x)$ is the only polynomial representation of $f(q)$ in $S_q$.
\end{corollary}

\begin{proof}
	Write $q=n/d$ for some relatively prime $n,d \in \nn$. Since $q \notin \nn \cup \nn^{-1}$, we see that $\min\{n,d\} \ge 2$. Let $f(x) \in \nn_0[x]$ be a polynomial having all coefficients equal to $0$ or $1$. First, assume that $q<1$ and so that $n<d$. Then the reduction process in the proof of Proposition~\ref{prop:optimal polynomial representations} sends every polynomial representation $h(x) \in \nn_0[x]$ of $f(q)$ to the unique optimal one, which is $f(x)$ by Proposition~\ref{prop:optimal polynomial representations} (indeed, $[x^i]f(x) \le 1 < d$ for every $i \in \nn_0$). If a nontrivial sequence of reductions ended at $f(x)$ then its final step would be a replacement of the form $dx^j \mapsto nx^{j-1}$ and so $[x^{j-1}]f(x) \ge n \ge 2$, which is impossible because every coefficient of $f(x)$ is at most~$1$. The case $q>1$ is analogous, using the replacements $nx^j \mapsto dx^{j+1}$ and the inequality $d \ge 2$.
\end{proof}

\bigskip
\section{The Monoid of Technical Fractions}
\label{sec:monoid T_q}

It turns out that every monogenic semidomain has a divisor-closed submonoid that encodes relevant information about its arithmetic and factorization behavior. In this section, we will explore such submonoids.

Let us start by proving the following lemma.

\begin{lem} \label{lem:lower bound on divisor}
	For $q \in (0,1) \cap \qq$ such that $q$ is not a unit fraction, let $S_q$ be the rational monogenic semidomain generated by $q$. For each nonzero $r \in S_q$, there exists $m \in \nn$ such that every divisor of $r$ in $S_q$ belongs to the interval $(q^m, r/q^m)$.
\end{lem}

\begin{proof}
	Set $n := \num(q)$ and $d := \den(q)$. Since $q$ is not a unit fraction, $n>1$. 
	
	Fix a prime $p$ such that $p \mid n$. Then $p \nmid d$ as $\gcd(n,d) = 1$. In particular, every nonzero element of~$S_q$ has nonnegative $p$-adic valuation: indeed, if $q' \in S_q^*$ then the denominator of $q'$ divides a power of $d$, and so $p \nmid \mathsf{d}(q')$. 
	
	Now fix a nonzero $r \in S_q$, and let $\mathcal{D}(r)$ be the set consisting of all nonzero divisors of $r$ in $S_q$. Then take $m \in \nn$ such that the following inequality holds:
	\begin{equation} \label{eq:the exponent m}
		v_p(q^m) > v_p(r).
	\end{equation}
	We are done once we prove that $\mathcal{D}(r) \subset (q^m, r/q^m)$. For this, fix $s \in \mathcal{D}(r)$, and let us argue that $s \in (q^m, r/q^m)$. Let $f_s(x) \in \nn_0[x]$ be a polynomial representation of $s$ in $S_q$. Observe that
	\begin{align*}
		v_p(s) = v_p(f_s(q)) &\ge \min \big\{ v_p([x^j]f_s) + jv_p(q) : j \in \ldb 0, \deg f_s \rdb \big\} \\
		            &\ge \min \big\{ v_p(q^j) : j \in \ldb 0, \deg f_s \rdb \big\} \\
					&= v_p(q^{\text{ord} \, f_s}).
	\end{align*}
	Before our final argument, we need to establish the following claim.
	\smallskip
	
	\noindent \textsc{Claim.} $\text{ord} \, f_s \le m$. 
	\smallskip
	
	\noindent \textsc{Proof of Claim.} Assume, towards a contradiction,
	that $\text{ord} \, f_s > m$. Then
	\begin{equation}\label{eq:val_s less than val_r}
		v_p(s) \ge v_p(q^{\text{ord} \, f_s}) > v_p(q^m) > v_p(r).
	\end{equation}
	On the other hand, since $v_p(q') \ge 0$ for all $q' \in S_q$, the $p$-valuation of $r/s$ is nonnegative, whence $v_p(r) - v_p(s) = v_p(r/s) \ge 0$. Hence $v_p(r) \ge v_p(s)$, which contradicts the inequality we have already obtained in~\eqref{eq:val_s less than val_r}. Hence $\text{ord} \, f_s \le m$, and the claim is now established.
	\smallskip
	
	We are in a position to show that $s \in (q^m, r/q^m)$. As $s \neq 0$, the polynomial $f_s(x)$ is not the zero polynomial. Therefore we can pick coefficients $c_0, \dots, c_{\deg f_s} \in \nn_0$ not all zeros such that
	\[
		f_s(x) = \sum_{j=0}^{\deg f_s} c_j x^j
	\]
	 According to the established claim, $\text{ord} \, f_s < m$. This, along with the fact that $q \in (0,1)$, ensures that
	\[
		s = f_s(q) = \sum_{j= \text{ord} \, f_s}^{\deg f_s} c_j q^j \ge q^{\text{ord} \, f_s} > q^m.
	\]
	Since the cofactor $r/s$ also belongs to $\mathcal{D}(r)$, we can similarly show that $r/s > q^m$, which implies that $s < r/q^m$. Therefore both inequalities $q^m < s$ and $s < r/q^m$ hold, which implies that $s \in (q^m, r/q^m)$. Hence we can conclude that $q^m < s < r/q^m$ for all $s \in \mathcal{D}(r)$.
\end{proof}

\begin{corollary}
	For $q \in \qq_{>0}$ such that $q$ is not a unit fraction, let $S_q$ be the rational monogenic semidomain generated by $q$. For each nonzero $r \in S_q$, there exist $\ell,u \in \rr_{>0}$ with $\ell<u$ such that every divisor of $r$ in $S_q$ belongs to the interval $(\ell,u)$.
\end{corollary}

\begin{proof}
	Fix a nonzero $r \in S_q$, and let $\mathcal{D}(r)$ be the set consisting of all nonzero divisors of $r$ in $S_q$. If $q<1$ then Lemma~\ref{lem:lower bound on divisor} ensures the existence of $m \in \nn$ such that $q^m < r/q^m$ and $\mathcal{D}(r) \subset (q^m, r/q^m)$. Thus, we can assume that $q>1$. In this case, $1 = \min S_q^*$ and so $1/2 < s \le r = \max \mathcal{D}(r)$. Hence after we take $\ell := 1/2$ and $u := r+1$, the inclusion $\mathcal{D}(r) \subset (\ell, u)$ holds.
\end{proof}

\medskip
\subsection{The Monoid of Technical Fractions}

Before establishing our next result, it is convenient to introduce, for each $q \in \qq_{>0}$, a monoid homomorphism $\omega_q \colon S_q^* \to \nn_0$ and a submonoid $T_q$ of $S_q^*$, which we define as follows:
\begin{equation} \label{eq:omega_q}
	\omega_q(r) := \sum_{\substack{p \in \pp \, : \, p \nmid \den(q)}} v_p(r)
\end{equation}
for all nonzero $r \in S_q$, while
\begin{equation}
	T_q := \omega_q^{-1}(0) = \big\{ r \in S_q^* : v_p(r) = 0 \text{ for all } p \nmid \mathsf{d}(q) \big\}.
\end{equation}
Observe that, for each $q \in \qq_{>0}$, the subset $T_q$ of $S_q$ consists of all the nonzero fractions whose numerators and denominators, when written in lowest terms, are supported only on the prime factors of $\mathsf{d}(q)$. Let us take a look at an example.

\begin{exam}
	Fix $q := 5/6$, and consider the rational monogenic semidomain $S_q$. In this case, $\mathsf{d}(q)=6$, and so the only primes that do not appear in the sum defining $\omega_q$ are $2$ and $3$. Thus, for every nonzero $r \in S_q$, if we write $r$ in lowest terms, say $r = \frac{m}{2^j 3^k}$ for some $m \in \nn$ and $j,k \in \nn_0$ then
	\[
		\omega_q(r) = \sum_{p \in \pp \setminus \{2,3\}} v_p(r) = \sum_{p \in \pp \setminus \{2,3\}} v_p(m).
	\]
	In other words, $\omega_q(r)$ counts, with multiplicity, the prime factors of the numerator of $r$ different from $2$ and $3$. Consequently,
	\[
		T_q = \{ r \in S_q^* : \num(r) = 2^j 3^k \text{ for some } j,k \in \nn_0 \}.
	\]
	For instance, every natural number of the form $2^j 3^k$ belongs to $T_q$. In addition, $T_q$ contains nonintegral elements of $S_q$:
	\[
		1+2q=\frac{8}{3} \in T_q, \quad 2+3q=\frac{9}{2} \in T_q, \quad \text{and} \quad 3+2q+3q^2=\frac{27}{4} \in T_q.
	\]
	On the other hand, there are elements of $S_q^*$ that do not belong to $T_q$: for instance, $q=5/6 \in S_q^*$ does not belong to $T_q$ because $v_p(q) = 0$ for every $p \in \pp \setminus \{2,3,5\}$ and so
	\[
		\omega_q(q) = \sum_{p \in \pp \setminus \{2,3\}} v_p(q) = v_5(q) = 1 \neq 0.
	\]
	\hfill $\blacksquare$
\end{exam}

Let us argue the following proposition.

\begin{prop} \label{prop:submonoid of singular fractions}
	For $q \in \qq_{>0}$, let $S_q$ be the monogenic semidomain generated by $q$, and let $\omega_q \colon S_q^* \to \nn_0$ be as defined in~\eqref{eq:omega_q}. Then the following statements hold.
	\begin{enumerate}
		\item $\omega_q$ is a monoid homomorphism.
		\smallskip
		
		\item $T_q = S_q^* \cap \gp(p \in \pp : p \mid \mathsf{d}(q))$.
		\smallskip
		
		\item $T_q$ is a divisor-closed submonoid of $S^*_q$.
	\end{enumerate}
\end{prop}

\begin{proof}
	Set $n := \mathsf{n}(q)$ and $d := \mathsf{d}(q)$. In addition, let $P$ be the subset of $\pp$ consisting of all the prime factors of~$d$.
	\smallskip
	
	(1) For each $p \in \pp \setminus P$, the valuation map $v_p \colon \qq^\times \to \zz$ is a well-defined group homomorphism from the group of units $\qq^\times$ of the field $\qq$ to the additive abelian group $\zz$, so the map $\omega \colon \qq^\times \to \zz$ defined as $\omega(r) := \sum_{p \in \pp \setminus P} v_p(r)$ is also a well-defined group homomorphism: indeed,  
	\[
		\omega(r_1 r_2) = \sum_{p \in \pp \setminus P} v_p(r_1 r_2) = \sum_{p \in \pp \setminus P} \big(v_p(r_1) + v_p(r_2)\big) = \omega(r_1) + \omega(r_2)
	\]
	for all $r_1, r_2 \in \qq^\times$. As $\omega_q \colon S_q^* \to \nn_0$ is the restriction map of $\omega$ to the multiplicative monoid $S_q^*$, it must be a monoid homomorphism.
	\smallskip
	
	(2) Let $\gp(P)$ denote the subgroup of $\qq^\times$ generated by $P$, which is the free (multiplicative) group on~$P$. For each $t \in T_q$ both $\num(t)$ and $\den(t)$ are finite products of primes in $P$, whence $T_q \subseteq S_q^* \cap \gp(P)$. For the reverse inclusion, it suffices to observe that if $t \in S_q^* \cap \gp(P)$ then the fact that $t \in \gp(P)$ guarantees that $v_p(t) = 0$ holds for any $p \in \pp \setminus P$ and so $\omega_q(t) = \sum_{p \in \pp \setminus P} v_p(t) = 0$, so the fact that $t \in S_q^*$ ensures that $t \in T_q$. Hence the inclusion $S_q^* \cap \gp(P) \subseteq T_q$ also holds.
	\smallskip
	
	(3) As $v_p(1) = 0$ for every $p \in \pp$, it follows that $\omega_q(1) = \sum_{p \in \pp \setminus P} v_p(1) = 0$ and so $1 \in T_q$. In addition, for all $t_1, t_2 \in T_q$, the fact that $\omega_q$ is a monoid homomorphism, along with the fact that $\omega_q(t_1) = \omega_q(t_2) = 0$, guarantees that $\omega_q(t_1 t_2) = \omega(t_1) + \omega(t_2) = 0$. Hence $T_q$ is a submonoid of $S_q^*$. To show that $T_q$ is divisor-closed, suppose that $s$ divides $t$ in $S_q^*$ for some $s \in S_q^*$ and $t \in T_q$. Then we can take $r \in S_q^*$ such that $t = rs$ and, after applying $\omega_q$, we obtain $0 = \omega_q(t) = \omega_q(rs) = \omega_q(r) + \omega_q(s)$, whence $\omega_q(r) = \omega_q(s) = 0$ as $\omega_q(r), \omega_q(s) \in \nn_0$. From $\omega_q(s) = 0$ we deduce that $s \in T_q$. Hence $T_q$ is a divisor-closed submonoid of $S_q^*$.
\end{proof}

Given that the monoid homomorphism $\omega_q$ and the submonoid $T_q$ of $S_q^*$ play a crucial role in the rest of this paper, it is convenient to introduce some terminology to refer to them.

\begin{defn}
	For $q \in \qq_{>0}$, let $S_q$ be the monogenic semidomain generated by~$q$. We call the rationals in the multiplicative submonoid $T_q$ of $S_q^*$ the \emph{technical fractions} of $S_q$, while we call the monoid $T_q$ the \emph{technical monoid} of $S_q$.
\end{defn}

Before proceeding further, let us determine the monoid of technical fractions of $S_q$ when the generator $q$ is a positive integer or a unit fraction.

\begin{example} \label{ex:T_q when q is a positive integer}
	For $q \in \nn$, let us show that $T_q$ is the trivial monoid. As $q \in \nn$, we see that $S_q = \nn_0$ and $\mathsf{d}(q)=1$. Therefore the set of prime factors of $\mathsf{d}(q)$ is empty, and Proposition~\ref{prop:submonoid of singular fractions}(2) yields
		\[
			T_q = S_q^* \cap \gp(p \in \pp : p \mid \mathsf{d}(q))=\nn \cap \{1\}=\{1\}.
		\]
		Equivalently, for every $r \in \nn$, the value $\omega_q(r)$ is the total number of prime factors of $r$, counted with multiplicity, so $\omega_q(r)=0$ if and only if $r=1$.
		\hfill $\blacksquare$
\end{example}

\begin{example} \label{ex:T_q when q is a unit fraction}
	We determine $T_q$ when $q$ is a unit fraction different from~$1$. Take $q := 1/d$ for some $d \in \nn_{\ge 2}$, and let $P$ be the set of prime factors of~$d$. Then Proposition~\ref{prop:submonoid of singular fractions}(2) yields
		\[
			T_q = S_q^* \cap \gp(P).
		\]
		For each $p \in P$, both $p$ and $1/p=(d/p)q$ belong to $S_q$, and so every element of $\gp(P)$ is a unit of $S_q$. Hence $\gp(P) \subseteq S_q^*$, which implies that
		\[
			T_q = \gp(P)=\big\{p_1^{e_1}\cdots p_\ell^{e_\ell} : e_1, \dots, e_\ell \in \zz \big\},
		\]
		where $p_1,\dots,p_\ell$ are the distinct prime divisors of~$d$. In particular, when $q$ is a unit fraction, every technical fraction is a unit of~$S_q$.
		\hfill $\blacksquare$
\end{example}

Before continuing with the monoid of technical fractions, let us argue the following lemma, which we need at our disposal for the proof of the next proposition.

\begin{lem} \label{lem:denseness of powers}
	For any $a,b \in \nn_{\ge 2}$ with $\gcd(a,b) = 1$, the set $\{a^k/b^\ell : k,\ell \in \nn_0\}$ is dense in $\rr_{>0}$.
\end{lem}

\begin{proof}
	First, observe that $\log a/\log b$ is irrational. Indeed, if $\log a/\log b=m/n$ for some $m,n \in \nn$ then $a^n=b^m$, which is not possible due to the fact that $\gcd(a,b)=1$ and both $a$ and $b$ are greater than~$1$. Now set $\alpha:=\log a$ and $\beta:=\log b$. Since $\alpha/\beta$ is irrational, the set
		\[
			\Big\{ k \frac{\alpha}\beta - \Big\lfloor k\frac{\alpha}\beta \Big\rfloor : k \in \nn_0 \Big\}
		\]
		is dense in the interval $[0,1]$. Now fix $\rho \in \rr$ and $\epsilon>0$, and set $y:=\rho/\beta$. Choose $k \in \nn_0$ sufficiently large so that
		\[
			\Big\lfloor k\frac{\alpha}{\beta}\Big\rfloor \ge \lfloor y \rfloor
			\quad \text{and} \quad
			\Big|\Big(k\frac{\alpha}{\beta}-\Big\lfloor k\frac{\alpha}{\beta}\Big\rfloor\Big)-(y-\lfloor y \rfloor)\Big|<\frac{\epsilon}{\beta}.
		\]
		Then $\ell:=\lfloor k\alpha/\beta\rfloor-\lfloor y \rfloor$ belongs to $\nn_0$, and
		\[
			\big|k\alpha-\ell\beta-\rho\big|<\epsilon.
		\]
		Hence the set $\{k \log a - \ell \log b : k,\ell \in \nn_0\}$ is dense in $\rr$. After exponentiating, we obtain that the set $\{a^k/b^\ell : k,\ell \in \nn_0\}$ is dense in $\rr_{>0}$.
\end{proof}

The following proposition will be helpful in the next sections because the infimum of $T_q \setminus \{1\}$ encodes some arithmetic aspects of $S_q^*$.

\begin{prop}\label{lem:inf T > 1}
	For $q \in \qq_{>0}$ such that $q$ is not a unit fraction, the following statements hold.
	\begin{enumerate}
		\item $\inf T_q \ge 1$.
		\smallskip
		
		\item $\inf (T_q \setminus \{1\}) = 1$ if $q \in (0,1)$ and $\den(q)$ is not a prime power.
	\end{enumerate}
\end{prop}

\begin{proof}
	(1) It suffices to show that no element of $T_q$ is smaller than~$1$. This is clear when $q \ge 1$, so suppose that $q<1$. Write $q=n/d$ in lowest terms. Since $q$ is not a unit fraction, $n>1$, and so we can fix $p \in \pp$ such that $p \mid n$. Then $p \nmid d$. Now take $r \in S_q^*$ with $r<1$, and let $f(x) \in \nn_0[x]$ be a polynomial representation of $r$. Since $q<1$, the inequality $r<1$ forces $\text{ord} \, f \ge 1$. For every $j \in \ldb \text{ord} \, f, \deg f \rdb$ with $[x^j]f(x) \ne 0$, we have
	\[
		v_p(([x^j]f(x))q^j) = v_p([x^j]f(x)) + jv_p(q) \ge (\text{ord} \, f)v_p(q) \ge v_p(q) \ge 1.
	\]
	Thus, the $p$-adic valuation of every nonzero term of $f(q)$ is at least~$1$, and so $v_p(r)=v_p(f(q)) \ge 1$. Since $p \nmid d$, the prime $p$ appears in the sum defining $\omega_q$, and therefore $\omega_q(r) \ge v_p(r) \ge 1$. Hence $r \notin T_q$. Thus, $T_q \subseteq \qq_{\ge 1}$, which implies that $\inf T_q \ge 1$.
	\smallskip
	
	(2) Now assume that $q \in (0,1)$ and that $\den(q)$ is not a prime power. Write $q=n/d$ in lowest terms. Since $q$ is not a unit fraction, $n \ge 2$. Set $B_q:=B_{q,\infty}$, the constant appearing in Corollary~\ref{cor:Frobenius bound for q<1}. Since $d$ is not a prime power, there are distinct primes $p_1$ and $p_2$ dividing~$d$. Fix $k \in \nn$. Since $p_1$ and $p_2$ are both coprime to $n^k$, their orders modulo $n^k$ are well defined; set
	\[
		\alpha_k := \text{ord}_{n^k}(p_1) \quad \text{and} \quad \beta_k := \text{ord}_{n^k}(p_2).
	\]
	Then $p_1^{\alpha_k} \equiv p_2^{\beta_k} \equiv 1 \pmod{n^k}$. Because $\gcd(p_1^{\alpha_k}, p_2^{\beta_k}) = 1$, Lemma~\ref{lem:denseness of powers} ensures that the set
	\[
		\big\{p_1^{i\alpha_k}/p_2^{j\beta_k} : i,j \in \nn_0 \big\}
	\]
	is dense in $\rr_{>0}$. Hence we can choose $i,j \in \nn_0$ such that
	\[
		r_k:=p_1^{i\alpha_k}/p_2^{j\beta_k} \in (1 + B_q q^k, 1 + 2B_q q^k).
	\]
	Setting $\epsilon_k:=r_k-1$, we see that
	\[
		B_q < \frac{\epsilon_k}{q^k} = \frac{r_k - 1}{q^k} < 2B_q.
	\]
	Also, the congruences above imply that $n^k \mid p_1^{i\alpha_k} - p_2^{j\beta_k}$. Therefore
	\[
		u_k := \frac{\epsilon_k}{q^k} = \frac{d^k}{n^k}(r_k - 1) = \frac{d^k}{n^k} \Big(\frac{p_1^{i\alpha_k}}{p_2^{j\beta_k}} - 1 \Big) = \frac{d^k \big(p_1^{i\alpha_k} - p_2^{j\beta_k} \big)}{n^k p_2^{j\beta_k}}
	\]
	has denominator divisible only by primes dividing $d$. Moreover, $u_k>B_q$. By Corollary~\ref{cor:Frobenius bound for q<1}, it follows that $u_k \in S_q$, and so $r_k = 1 + q^k u_k \in S_q$. Since both the numerator and denominator of $r_k$ are supported on primes dividing $d$, we have $\omega_q(r_k)=0$. Hence $r_k \in T_q \setminus \{1\}$. Finally, the inequalities $1+B_q q^k < r_k < 1+2B_q q^k$ imply that $r_k \to 1$ as $k \to \infty$. Together with part~(1), this shows that $\inf (T_q \setminus \{1\})=1$.
\end{proof}

We conclude this section with the following example.

\begin{exam}
		Fix $q=5/6$, and consider the rational monogenic semidomain $S_q$. As $\den(q)=6$ has two distinct prime divisors, the previous proposition applies. The technical fractions of $S_q$ are precisely the elements of $S_q^*$ whose numerators and denominators, in lowest terms, are supported only on the primes $2$ and $3$. The proposition says that, despite the fact that no technical fraction is less than~$1$, there are technical fractions different from~$1$ and arbitrarily close to $1$ from above. This contrasts with the prime-power denominator case, such as $q=2/3$, where the preceding construction with two distinct denominator primes is unavailable.
		\hfill $\blacksquare$
\end{exam}

\bigskip
\section{Factorization}
\label{sec:factorization}
\label{sec:BF}

In this section we start our investigation of factorization in the class consisting of all rational monogenic semidomains. We will consider all the properties shown in the diagram of Figure~\ref{fig:AAZ subdiagram for monogenic semidomains}.
\begin{center}
	\begin{figure}
		\begin{tikzcd}[cramped]
			\textbf{ UF } \ \arrow[r, Rightarrow, shift left=-0.3ex] \arrow[red, r, Leftarrow, "/"{anchor=center,sloped}, shift left=1.2ex] \arrow[d, Rightarrow, shift right=1.3ex] \arrow[red, d, Leftarrow, "/"{anchor=center,sloped}, shift left=0.8ex] &  \ \textbf{ HF } \arrow[d, Rightarrow, shift right=0.6ex] \arrow[red, d, Leftarrow, "/"{anchor=center,sloped}, shift left=1.3ex] \\
			\textbf{ FF } \ \arrow[r, Rightarrow, shift left=-0.3ex]	\arrow[red, r, Leftarrow, "/"{anchor=center,sloped}, shift left=1.3ex] & \textbf{ BF }  \arrow[r, Rightarrow, shift left=-0.3ex] \arrow[red, r, Leftarrow, "/"{anchor=center,sloped}, shift left=1.3ex]  & \textbf{ACCP}.
		\end{tikzcd}
		\caption{The diagram shows the strenght of the factorization and ideal-theoretical properties we study in this section.}
		\label{fig:AAZ subdiagram for monogenic semidomains}
	\end{figure}
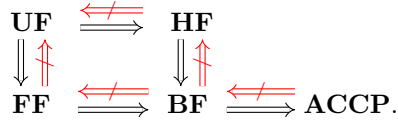
\end{center}
\medskip

\medskip
\subsection{The Factorial and Half-factorial Properties}
\label{sec:factoriality}

We start by determine the positive rational generators $q$ such that the rational monogenic semidomain $S_q$ has the UF property. First, let us establish the following lemma.

\begin{lemma} \label{lem:if q is integer or unit fraction S_q is UFS}
	For any $q \in \nn \cup \nn_{\ge 2}^{-1}$, let $S_q$ be the rational monogenic semidomain semidomain generated by $q$.
\end{lemma}

\begin{proof}
	Fix $q \in \nn \cup \nn^{-1}$. If $q \in \nn$ then $S_q = \nn_0$, which is a UFS by the Fundamental Theorem of Arithmetic. Otherwise, $q$ is a unit fraction of the form $q = 1/d$ for some $d \in \nn$ with $d \ge 2$. Then it follows from Proposition~\ref{prop:units of N_0[q]} that $S_q^\times$ is a free abelian group $\gp(P)$ on $P$, where
	\[
		P := \big\{ p \in \pp : p \mid d \big\}.
	\]
	Therefore we can write $S_q^* = S_q^\times F$, where $F$ is the submonoid generated by the set of rational primes $\pp \setminus P$ inside the multiplicative monoid $\nn$. As $\pp \setminus P$ is a set of non-associate primes elements inside the monoid $F$, we obtain that $F$ is a free commutative monoid. This, along with the fact that $S_q^* = S_q^\times F$, ensures that $S_q^*$ is a UFM, whence $S_q$ is a UFS.
\end{proof}

It turns out that $S_q$ has the UF property if and only if $S_q$ has the HF property, which happens precisely when $q$ is either a positive integer or a unit fraction.

\begin{thm} \label{thm:UF/HF}
	For $q \in \qq_{>0}$, let $S_q$ be the rational monogenic semidomain generated by $q$. Then the following conditions are equivalent.
	\begin{itemize}
		\item[(a)] $S_q$ is a UFS.
		\smallskip
		
		\item[(b)] $S_q$ is an HFS.
		\smallskip
		
		\item[(c)] $q \in \nn \cup \nn^{-1}$.
	\end{itemize}
\end{thm}

\begin{proof}
	(a) $\Rightarrow$ (b): This follows immediately, since every UFS is an HFS.
	\smallskip
	
	(b) $\Rightarrow$ (c): Suppose now that $S_q$ is an HFS. Assume, for the sake of a contradiction, that $q \notin \nn \cup \nn^{-1}$. Write $q=n/d$ for some relatively prime $n,d \in \nn$. As  $q \notin \nn \cup \nn^{-1}$, it follows that $\min\{n,d\} \ge 2$, and so $S_q^*$ is reduced by Proposition~\ref{prop:units of N_0[q]}.
	
	Consider the polynomial $R(x) := x^8 + 2x^7 + 2x^6 + 2x^5 + x^4 + x^3 + x^2 + x + 1$ in $\nn_0[x]$ and set $r := R(q) \in S_q$. In order to produce the desired contradiction, it suffices to argue that the element $r $ has factorizations of length $2$ and $3$ in $S_q$. Before exhibiting such factorizations, it is convenient to argue the following claim.
	\smallskip

	\noindent \textsc{Claim.} If an irreducible polynomial $A(x) \in \nn_0[x]$ is the unique polynomial representation of $A(q)$ in $S_q$ then $A(q)$ is an atom in $S_q$.
	\smallskip

	\noindent \textsc{Proof of Claim.} Let $A(x) \in \nn_0[x]$ be an irreducible polynomial in $\nn_0[x]$ such that $A(x)$ is the only polynomial representation of $A(q)$ in $S_q$. To argue that $A(q)$ is an atom, write $A(q) = uv$ for some $u,v \in S_q^*$. Let the polynomials $U(x)$ and $V(x)$ in $\nn_0[x]$ be polynomial representations of $u$ and $v$, respectively. After evaluating the polynomial $U(x)V(x)$ at $q$ one obtains  $A(q)$, whence $U(x)V(x)$ is also a polynomial representation of $A(q)$. Thus, the uniqueness of $A(x)$ as polynomial representation of $A(q)$ now ensures that $A(x) = U(x)V(x)$. As $A(x)$ is irreducible in $\nn_0[x]$, it follows that either $U(x)$ or $V(x)$ is the containt polynomial $1$. Therefore $u = U(q) = 1$ or $v = V(q) = 1$, and we conclude that $A(q)$ is an atom of $S_q$. Thus, the claim is established.
	\smallskip

	We are in a position to exhibit factorizations of $r$ in $S_q$ of length $2$ and $3$, which will give us the desired contradiction because the semidomain $S_q$ is an HFS.
	
	Let us first produce the length-$2$ factorization. For this, observe that
	\[
		A_1(x) := x^2 + x + 1 \quad \text{ and } \quad A_2(x) := x^6 + x^5 + x^3 + 1
	\]
	are polynomials in $\nn_0[x]$ that satisfy $R(x) = A_1(x)A_2(x)$, whence it suffices to argue that $A_1(q)$ and $A_2(q)$ are irreducible in $S_q$. Before proving so, it is convenient to argue show that $A_1(x)$ and $A_2(x)$ are irreducible polynomials in $\nn_0[x]$. 
	
	The polynomial $A_1(x)$ is irreducible in $\nn_0[x]$ because its value at $1$ is a rational prime and it has constant term~$1$. To argue that the polynomial $A_2(x)$ is also irreducible in $\nn_0[x]$, take polynomials $G(x)$ and $H(x)$ in $\nn_0[x]$ such that $A_2(x) = G(x)H(x)$ and further assume that $1 \notin \{G(x),H(x)\}$. Since the only positive integer dividing $A_2(x)$ is $1$, neither $G(x)$ nor $H(x)$ is a constant polynomial. Moreover, as $G(0)H(0) = A_2(0) = 1$, it follows that $G(0) = H(0) = 1$. Now the fact that $A_2(1)=4$ implies that $G(1)=H(1)=2$. Thus, both $G(x)$ and $H(x)$ are binomials with constant term~$1$. However, no product of two such binomials has support $\{0,3,5,6\}$. As a consequence, the polynomial $A_2(x)$ is irreducible in $\nn_0[x]$.

	As $q \notin \nn \cup \nn^{-1}$ and all the coefficients of $A_1(x)$ and $A_2(x)$ are either $0$ or $1$, it follows from Corollary~\ref{cor:unique polynomial representation if all coefficients are zero or one} that $A_1(x)$ and $A_2(x)$ are the unique polynomial representations of $A_1(q)$ and $A_2(q)$ in $S_q$, respectively. This, along with the fact that $A_1(x)$ and $A_2(x)$ are irreducibles in $\nn_0[x]$, allows us to conclude that $A_1(q)$ and $A_2(q)$ are atoms in $S_q$ (we are using here the established claim). As a result, the right-hand side of $r = A_1(q) A_2(q)$ is a length-$2$ factorization of~$r$ in~$S_q$.
	
	To produce the length-$3$ factorization of $r$ in $S_q$, we will proceed in a similar manner as we did for the length-$2$ factorization. First, observe that
	\[
		B_1(x) := 1+x, \quad B_2(x) := 1+x^2, \quad \text{and} \quad B_3(x) := 1+x^4+x^5
	\]
	are polynomials in $\nn_0[x]$ such that $R(x) = B_1(x)B_2(x)B_3(x)$. Then notice that the three polynomials are irreducible in $\nn_0[x]$ because their values at $1$ are prime and each has constant term~$1$. Because all the coefficients of $B_1(x)$, $B_2(x)$, and $B_3(x)$ are either $0$ or $1$, Corollary~\ref{cor:unique polynomial representation if all coefficients are zero or one} guarantees that $B_1(x)$, $B_2(x)$, and $B_3(x)$ are the unique polynomial representations of $B_1(q)$, $B_2(q)$, and $B_3(q)$ in $S_q$, respectively. Therefore $B_1(q)$, $B_2(q)$, and $B_3(q)$ are atoms of $S_q$ in light of the established claim. Hence it follows from the fact that $r = B_1(q)B_2(q)B_3(q)$ that $B_1(q)B_2(q)B_3(q)$ is a length-$3$ factorization of $r$ in $S_q$.

	As the element $r$ of $S_q$ has two factorization with different lengths, the semidomain $S_q$ is not an HFS, a contradiction. Hence $q \in \nn \cup \nn^{-1}$.
	\smallskip
	
	(c) $\Rightarrow$ (a): It follows from Lemma~\ref{lem:if q is integer or unit fraction S_q is UFS} that if the generator $q$ is either a positive integer or a unit fraction then $S_q$ is a UFS.
\end{proof}

The Bi-UF Positive Conjecture, mentioned in the introduction, states that the only positive semiring whose additive and multiplicative monoids both satisfy the UF property is the prototypical semidomain~$\nn_0$. Moreover, in the same paper where the Bi-UF Positive Conjecture was first proposed, the authors posed the following question.

\begin{quest} \cite[Question~7.8.1]{BCG21} \label{quest:bi-HFS}
	Is $\nn_0$ the only positive semiring whose additive and multiplicative monoids both satisfy the HF property?
\end{quest}

As an immediate consequence of Theorem~\ref{thm:UF/HF}, we can provide some evidence supporting not only the statement of the Bi-UF Positive Conjecture but also providing information correlated with a potential positive answer to Question~\ref{quest:bi-HFS}.

\begin{corollary}
	The only monogenic rational semidomain having the Bi-HF property is $\mathbb{N}_0$. Thus, the statement of the Bi-UF Positive Conjecture holds over the class of rational monogenic semidomains. 
\end{corollary}

\begin{proof}
	Consider the rational monogenic semidomain $S_q$ for some $q \in \qq_{>0}$, and assume that both the additive and the multiplicative monoids of $S_q$ are HFMs. In particular, the semidomain $S_q$ is an HFS and, therefore, it follows from Theorem~\ref{thm:UF/HF} that $q \in \nn \cup \nn^{-1}$. Observe that, for every unit fraction $1/d$ with $d \ge 2$, the additive monoid of $S_{1/d}$ consists of all $d$-adic rationals and so it is not even atomic: indeed, every nonzero $d$-adic rational $r$ can be written as the sum of $d$ copies of $r/d$, whence $r$ is not an additive atom of $S_{1/d}$. As a result, $q \notin \nn_{\ge 2}^{-1}$, which implies that $q \in \nn$. Hence $S_q = \nn_0[q] = \nn_0$. 
\end{proof}

\medskip
\subsection{The Finite Factorization Property}
\label{sec:FF}

This section is devoted to the FF property: we determine all rational parameters $q$ for which the semidomain $S_q$ is an FFS. The following well-known characterization for FFMs, established in~\cite{fHK92}, will be use in the main theorem of this section.

\begin{lem} \cite[Corollary~2]{fHK92} \label{lem:finite divisors imply FFM}
	A monoid $M$ is an FFM if and only if every element of $M$ has only finitely many divisors up to associates.
\end{lem}

There are three fundamentally different choices of the positive rational generator~$q$ that yield a rational monogenic semidomain $S_q$ that is an FFS. We describe these choices in the following example.

\begin{exam} \label{ex:the three cases wehre S_q is FFS}
	The subset of $\qq_{>0}$ consisting of all the positive rational values of the paramter $q$ such that the monogenic semidomain $S_q$ is an FFS can be conveniently expressed as the union of three (non-disjoint) subsets, each of them governed by a condition on $q$ that allows us to argue the FF property in a somehow natural way:
	\begin{enumerate}
		\item $q \ge 1$. For instance, when $q = 3/2$ the semidomain $S_q$ is an FFM because it is an increasing positive monoid of an ordered field (see \cite[Theorem~5.6]{fG19}).
		\smallskip
		
		\item $q$ is a unit fraction. For instance, for $q = 1/6$ the semidomain $S_q$ is a UFS and, therefore, an FFS (this will be argued in the next section).
		\smallskip
		
		\item $\den(q)$ is a prime power. For instance, for $q = 2/9$ the semidomain $S_q$ is an FFS even though the conditions in the previous two parts do not hold.
	\end{enumerate}
	Indeed, the previous three conditions are the typical cases one has to consider to determine the positive rational generators $q$ such that $S_q$ is an FFS. In particular, for rational generators in $(0,1)$ the semidomian $S_q$ is an FFS only in last two cases mentioned: either all denominator primes become units (as for the $q=1/6$ case) or the denominator of $q$ is supported at a single prime (as for the $q=2/9$ case).
	\begin{center}
	\renewcommand{\arraystretch}{1.2}
	\begin{tabular}{@{}c@{\qquad}l@{\qquad}c@{}}
		\hline
		$q \in \qq_{>0}$ & \multicolumn{1}{c}{Reason} & $S_q$ is an FFS \\ 
		\hline
		$3/2$ & $q \ge 1$ & yes \\
		$1/6$ & $1/q = 6 \in \nn_{\ge 2}$  & yes \\
		$2/9$ & $\den(q)=9$ is a prime power & yes \\
		\hline
		$5/6$ & $q$ \text{satisfies none of the above} & no \\
		\hline \hline
	\end{tabular}
	\end{center}
	\hfill $\blacksquare$
\end{exam}

We proceed to prove that the rational monogenic semidomain $S_q$ satisfies the FF property precisely in the first three cases indicated in Example~\ref{ex:the three cases wehre S_q is FFS}.

\begin{thm} \label{thm:FFS} \label{lem:non-prime-power denominator infinite divisors} \label{lem:prime-power-denominator finite divisors}
	For $q \in \qq_{>0}$, let $S_q$ be the rational monogenic semidomain generated by $q$. Then $S_q$ is an FFS if and only if at least one of the following conditions holds:
	\begin{itemize}
		\item $q \ge 1$,
		\smallskip
		
		\item $q$ is a unit fraction, or
		\smallskip
		
		\item $\mathsf{d}(q)$ is a prime power.
	\end{itemize}
\end{thm}

\begin{proof}
	For the direct implication, assume that $q$ does not satisfy any of the three conditions. Then~$q$ is a rational number inside the interval $(0,1)$ with $n := \mathsf{n}(q) \ge 2$ such that $d := \mathsf{d}(q)$ has at least two distinct prime factors, which allows us to pick distinct $p_1, p_2 \in \pp$ such that $p_1 \mid d$ and $p_2 \mid d$. By Lemma~\ref{lem:denseness of powers}, the set
	\[
		R := \big\{p_1^i/p_2^j : i,j \in \nn_0\big\}
	\]
	is dense in $\rr_{>0}$. Let $B_{q,\infty} = d(n-1)/(d-n)$ be the constant introduced in~\eqref{eq:Bq infinite} and choose $n_0 \in \nn$ such that $n_0 > B_{q,\infty}^2$. Then the open interval $(B_{q,\infty}, n_0/B_{q,\infty})$ contains infinitely many distinct elements of the set~$R$. For each $r \in R \cap (B_{q,\infty}, n_0/B_{q,\infty})$, every prime factor of the denominator of~$r$ and of the denominator of $n_0/r$ divides $d$, and
	\[
		\min\{r, n_0/r\} > B_{q,\infty}.
	\]
	After rewriting $r$ and $n_0/r$ with denominators that are powers of $d$, it follows from Corollary~\ref{cor:Frobenius bound for q<1} that both $r$ and $n_0/r$ belong to~$S_q^*$. Thus, every such $r$ is a divisor of the fixed element $n_0$ in $S_q^*$. Since there are infinitely many distinct such divisors and $S_q^*$ is reduced, the element $n_0$ has infinitely many divisors up to associates. Hence Lemma~\ref{lem:finite divisors imply FFM} implies that $S_q^*$ is not an FFM. As a consequence, the semidomain $S_q$ is not an FFS.
	\smallskip
	
	To prove the reverse implication, we assume that $q$ satisfies at least one of the three conditions in the statement of the theorem and then we split the rest of the proof into three cases, a case for each of the stated conditions.
	\smallskip
	
	\textsc{Case 1:} $q \ge 1$. Observe that the set $S_q \setminus \qq_{\ge m}$ is finite for every $m \in \nn$. Therefore we can list the elements of $S_q^*$ as a strictly increasing sequence $(s_n)_{n \ge 1}$. The map
	\[
		\ln \colon S_q^* \to \ln(S_q^*)
	\]
	is a monoid isomorphism from $S_q^*$ to the additive monoid $\ln(S_q^*)$. Thus, the monoid $\ln(S_q^*)$ is additively generated by the increasing sequence $(\ln s_n)_{n \ge 1}$, which implies that $\ln(S_q^*)$ is an increasing positive monoid of the ordered field~$\rr$. Hence $\ln(S_q^*)$ is an FFM by virtue of \cite[Theorem~5.6]{fG19}. Hence it follows that $S_q$ is an FFS.
	\smallskip
	
	\textsc{Case 2:} $q$ is a unit fraction. In this case, it follows from Lemma~\ref{lem:if q is integer or unit fraction S_q is UFS} that the semidomain $S_q$ is a UFS and, therefore, $S_q$ must be an FFS.
	\smallskip
	
	\textsc{Case 3:} $\mathsf{d}(q)$ is a prime power. Given that we have already proved the previous cases, we can assume, without loss of generality, that $q \in (0,1)$ and also that $q$ is not a unit fraction. Take $p \in \pp$ and then take $n,k \in \nn$ with $q = n/p^k$ such that $p \nmid n$. Since $q$ is not a unit fraction, $n \ge 2$. Then it follows from Proposition~\ref{prop:units of N_0[q]} that $S_q$ is a reduced semidomain.
	
	Let us prove that every element of the monoid $S_q^*$ has only finitely many divisors. For this, fix $r \in S_q^*$. By virtue of Lemma~\ref{lem:lower bound on divisor}, we can take $\ell,u \in \rr_{>0}$ with $\ell < u$ such that all divisors of $r$ in $S_q^*$ belong to the interval $(\ell,u)$. Let $s$ be a divisor of $r$ in $S_q^*$. Since $s \in S_q^*$, we can write $s = n_s/p^e$ for some $n_s \in \nn$ and $e \in \nn_0$. If $p' \in \pp \setminus \{p\}$ then $v_{p'}(s) \ge 0$, while the divisibility relation $s \mid_{S_q} r$ ensures that $v_{p'}(s) \le v_{p'}(r)$. Therefore the part of~$s$ supported away from $p$ has only finitely many possible values. In addition, for each such value, the condition $s \in (\ell,u)$ allows only finitely many possible powers of $p$. Hence $r$ has only finitely many divisors in $S_q^*$. Hence $S_q^*$ is an FFM in light of Lemma~\ref{lem:finite divisors imply FFM}. Hence the semidomain $S_q$ is an FFS.
\end{proof}

\medskip
\subsection{The Bounded Factorization Property}

We now consider the BF property. We prove that a rational monogenic semidomain is a BFS if and only if it satisfies the ACCP. We also provide a characterization in terms of the monoid of technical fractions. By convention, let us assume that the infimum of the empty set is infinite and so greater than any positive integer.

\begin{thm}\label{thm:bfs = accp = inf T > 1}
	For $q \in \qq_{>0}$ consider the monogenic semidomain $S_q$ generated by $q$. If $q$ is not a unit fraction then the following conditions are equivalent.
	\begin{enumerate}
		\item[(a)] $S_q$ is a BFS.
		\smallskip
		
		\item[(b)] $S_q$ satisfies the ACCP.
		\smallskip
		
		\item[(c)] $\inf (T_q \setminus \{1\}) > 1$.
	\end{enumerate}
\end{thm}

\begin{proof}
	First, assume that $q=1$ and, therefore, that $S_q = \nn_0$. Thus, $S_q$ trivially satisfies the BF and the ACCP properties. In addition, we have seen in Example~\ref{ex:T_q when q is a positive integer} that the monoid of technical fractions of $\nn_0$ is trivial and so $T_q \setminus \{1\}$ is the empty set. Hence $\inf(T_q \setminus \{1\}) = \infty > 1$.

	Now assume that $q>1$ and let us argue that in this case the three conditions also hold. Every element of $S_q^* \setminus \{1\}$ is at least $\min\{2,q\}$, whence
	\[
		\inf (T_q \setminus \{1\}) \ge \inf (S_q^* \setminus \{1\}) \ge \min\{2,q\}>1.
	\]
	On the other hand, we have already seen in the proof of Theorem~\ref{thm:FFS} that $S_q$ is an FFS provided that $q>1$, and this implies that $S_q$ is a BFS and, as a consequence, that $S_q$ satisfies the ACCP.
	\smallskip

	Therefore we may assume that $q \in (0,1)$ for the rest of the proof. Since $q$ is not a unit fraction, Proposition~\ref{prop:units of N_0[q]} ensures that the semidomain $S_q$ is reduced.
	\smallskip

	(a) $\Rightarrow$ (b): This follows from the standard fact that every BFM satisfies the ACCP.
	\smallskip

	(b) $\Rightarrow$ (c): Let $B_q := B_{q, \infty}$ be the constant introduced in \eqref{eq:Bq infinite}. Suppose, by way of contradiction, that $\inf(T_q \setminus \{1\})=1$. As $q$ is not a unit fraction, in light of Proposition~\ref{lem:inf T > 1}(1) we can choose a sequence $(t_n)_{n \ge 1}$ in $T_q \setminus \{1\}$ such that
	\[
		1<t_n<2^{1/2^n}
	\]
	for every $n \in \nn$, and then set $P_n := t_1\cdots t_n$ for each $n \in \nn_0$ (as $P_0$ is the empty product, $P_0=1$). For each $n \in \nn$,
	\[
		P_n = \prod_{j=1}^n t_j < \prod_{j=1}^n 2^{1/2^j} = 2^{1-1/2^n} < 2.
	\]
	Now fix $N \in \nn$ such that $N > 2B_q$, and observe that $N/P_n > N/2 > B_q$ for every $n \in \nn_0$.

	We claim that $N/P_n \in S_q^*$ for every $n \in \nn_0$. This is clear when $n=0$. Now fix $n \in \nn$, and observe that, for every $j \in \ldb 1,n \rdb$, the fact that $t_j \in T_q$ implies that every prime factor of $\num(t_j)\den(t_j)$ divides $\den(q)$. Hence the denominator of $N/P_n$ divides a power of $\den(q)$. If we write $N/P_n=c/\den(q)^k$ for some $c \in \nn$ and $k \in \nn_0$ then either $k=0$, in which case $N/P_n \in \nn \subseteq S_q^*$, or $k \in \nn$, in which case the inequality $N/P_n>B_q$ and Corollary~\ref{cor:Frobenius bound for q<1} guarantee that $N/P_n \in S_q^*$.

	For each $n \in \nn_0$, we can write
	\[
		\frac{N}{P_n}=\frac{N}{P_{n+1}}t_{n+1}.
	\]
	Thus, $N/P_{n+1}$ divides $N/P_n$ in $S_q^*$, and so the sequence $\big((N/P_n)S_q^*\big)_{n \ge 0}$ is an ascending chain of principal ideals of $S_q^*$. This chain ascends strictly at each step because $S_q^*$ is reduced and $t_{n+1} \ne 1$ for every $n \in \nn_0$. This contradicts that $S_q$ satisfies the ACCP. Hence $\inf (T_q \setminus \{1\}) > 1$.
		\smallskip

	(c) $\Rightarrow$ (a): Set $\delta:=\inf (T_q \setminus \{1\})$ and assume that $\delta>1$. We show that every strict ascending chain of principal ideals of $S_q^*$ has bounded length, which is enough to conclude that $S_q^*$ is a BFM. Let $(r_nS_q^*)_{n \ge 0}$ be a strict ascending chain of principal ideals of $S_q^*$. Thus, for each $n \in \nn_0$, we can take a nonunit $a_n \in S_q^*$ such that $r_n = a_n r_{n+1}$. After applying the monoid homomorphism $\omega_q$, we obtain that
	\[
		\omega_q(r_n) = \omega_q(a_nr_{n+1}) = \omega_q(a_n) + \omega_q(r_{n+1}) \ge \omega_q(r_{n+1}).
	\]
	Therefore $(\omega_q(r_n))_{n \ge 0}$ is a non-increasing sequence, and this implies that the number of indices $n$ for which $\omega_q(r_n)>\omega_q(r_{n+1})$ is at most $\omega_q(r_0)$.

	It remains to bound the number of indices for which $\omega_q(r_n)=\omega_q(r_{n+1})$. For such an index, the equality above forces $\omega_q(a_n)=0$, so $a_n \in T_q\setminus\{1\}$ and, as a result, $a_n \ge \delta$. Hence $r_{n+1}=r_n/a_n \le r_n/\delta$. By Lemma~\ref{lem:lower bound on divisor}, all divisors of $r_0$ in $S_q^*$ lie in some interval $[\ell,u]$ with $0<\ell<u$. Since each $r_n$ divides $r_0$, every $r_n$ lies in $[\ell,u]$. Consequently, any consecutive block of indices satisfying $\omega_q(r_n)=\omega_q(r_{n+1})$ has length at most $\log(u/\ell)/\log\delta+1$. Since there are at most $\omega_q(r_0)$ indices for which $\omega_q(r_n)>\omega_q(r_{n+1})$, the length of $(r_nS_q^*)_{n \ge 0}$ is bounded. Hence $S_q$ is a BFS.
\end{proof}

After the appearance of \cite{AAZ90}, the property of being atomic has been studied in connection with the factorization properties we investigated in this section. However, we were unable to determine whether there exists $q \in \qq$ such that $\nn_0[q]$ is not atomic. Therefore we conclude this section with the following question.

\begin{question}
	Is it possible to construct a rational monogenic semidomain that is not atomic?
\end{question}

\bigskip
\section{The Krull Property}
\label{sec:Krull property}

In this section, we first determine the rational monogenic semidomains whose multiplicative monoids are root-closed. Then we use this information to determine the rational monogenic semidomains that satisfy the Krull property. 

\medskip
\subsection{The Notion of Height}

We proceed to introduce the notion of height for elements of a rational monogenic semidomain. Fix a positive rational $q$ and then set $n := \num(q)$ and $d := \den(q)$. Let $S_q$ be the rational monogenic semidomain generated by $q$. Recall that $S_q$ is the nested union of the ascending chain $(M_{q,k})_{k \ge 0}$ of submonoids of $S_q$, where
\[
	M_{q,k} := \big\langle q^j : j \in \ldb 0,k \rdb \big\rangle
\]
for every $k \in \nn_0$. For each nonzero $r \in S_q$, we define the \emph{height} $h(r)$ of~$r$ as follows: $h(r) = 0$ if $r \in M_{q,0} = \nn_0$ and, otherwise, $h(r)$ is the unique $h \in \nn$ such that $r \in M_{q,h}$ but $r \notin M_{q,h-1}$. Let us take a look at the following example.

\begin{exam}
	For $q := 2/5$, we consider the rational monogenic semidomain $S_q$. We show that $r:=29/25$ belongs to $S_q$ and that $h(r) = 2$. We can deduce that $r \in S_q$ from the fact that $r = 29/25 = 1 + (2/5)^2$. To compute $h(r)$, first note that $r \notin M_{q,1} = \langle 1, 2/5 \rangle$ as $29/25$ cannot be written as a nonnegative integer combination of $1$ and $2/5$. Now observe that $M_{q,2} = \langle 1, 2/5, (2/5)^2 \rangle$, whence $r = 1 + (2/5)^2 \in M_{q,2}$. Hence $h(r)=2$.
	\hfill $\blacksquare$
\end{exam} 

In the following lemma we provide another description of the height.

\begin{lem} \label{lem:height characterization}
	For $q \in \qq_{>0}$, let $S_q$ be the rational monogenic semidomain generated by $q$. For each nonzero $r \in S_q$,
	\[
		h(r) = \min \Big\{ e \in \nn_0 : r = \frac{a}{\mathsf{d}(q)^e} \text{ for some } a \in \nn \Big\}.
	\]
\end{lem}

\begin{proof}
	Set $n := \num{q}$ and $d := \den(q)$. Fix a nonzero $r \in S_q$. If $d=1$ then $M_{q,k} = S_q = \nn_0$ for every $k \in \nn_0$ and so $h(r)=0$, from which we immediately obtain the desired equality. Thus, we can assume that $d>1$. Set
	\[
		E(r):= \Big\{ e \in \nn_0 : r = \frac{m}{d^e} \text{ for some } m \in \nn \Big\}.
	\]
	As $r \in S_q$, the set $E(r)$ is nonempty. Set $e(r) := \min E(r)$. We can readily check that $e(r) \le h(r)$: indeed, as $r \in M_{q,h(r)} = d^{-h(r)}N_{q,h(r)}$ and $r \neq 0$, we can take $m \in N_{q,h(r)}$ such that $r = m/d^{h(r)}$ for some $m \in \nn$. It remains to prove that $h(r) \le e(r)$. For this, we first pick a polynomial representation $f_r(x) \in \nn_0[x]$ of $r$ and write
	\begin{equation}\label{eq:polynomial rep of r}
		f_r(x) = \sum_{i=0}^{\deg f_r} c_i x^i
	\end{equation}
	for some $c_0, \dots, c_{\deg f_r} \in \nn_0$. We can further assume that the degree of $f_r(x)$ is as small as it can possibly be. In the following claim, we prove that $\deg f_r$ is a lower bound for the set $E(r)$.
	\smallskip
	
	\noindent \textsc{Claim.} $\deg f_r \le e(r)$.
	\smallskip
	
	\noindent \textsc{Proof of Claim.} Set $e_0 := \deg f_r$ and assume, towards a contradiction, that $e_0 > e$ for some $e \in E(r)$. Write $r = m/d^e$ for some $m \in \nn$. Given this assumption $e_0 > e$, we infer that $e_0 \ge 1$. Since $m/d^e = r = f_r(q)$, after multiplying both sides of~\eqref{eq:polynomial rep of r} by $d^{e_0}$ and evaluating at $x=q$, we obtain
	\[
		md^{e_0-e} = c_{e_0} n^{e_0} + \sum_{j=0}^{e_0 - 1} c_j n^j d^{e_0-j}.
	\]
	Reducing this equality modulo $d$ yields $d \mid c_{e_0} n^{e_0}$. As $\gcd(n,d) = 1$, it follows that $d \mid c_{e_0}$. Write $c_{e_0} = ds$ and replace the term $c_{e_0}x^{e_0}$ in $f_r(x)$ by $snx^{e_0-1}$ without changing the value at~$q$, and after doing so we obtain another polynomial representation of $r$ with smaller degree. However, this contradicts the minimality of the degree of $f_r(x)$. Hence the claim is established.
	\smallskip
	
	Finally, as $h(r)$ is the height of $r$, it follows that $r \in M_{q,h(r)} = \sum_{j=0}^{h(r)} \nn_0 q^j$ but $r \notin \sum_{j=0}^{h(r)-1} \nn_0 q^j$. Thus, no polynomial representation of $r$ has degree strictly less than $h(r)$. Therefore $h(r) \le \deg f_r$. On the other hand, it follows from the established claim that $\deg f_r \le e(r)$. Hence $h(r) \le \deg f_r \le e(r)$. This, along with the inequality $e(r) \le h(r)$, which we previously established, allows us to conclude that $h(r)=e(r)$.
\end{proof}

\medskip
\subsection{Root Closedness}

Fix $q \in \qq_{>0}$ and set $n := \num{q}$ and $d := \den{q}$. Before delving into the root closure of rational monogenic semidomains, let us recall the sequence $(B_{q,k})_{k \ge 1}$ and the constant $B_{q,\infty}$ introduced in Section~\ref{sec:algebraic aspects}: if $q \in \qq_{>0} \setminus (\nn \cup \nn_{\ge 2}^{-1})$ then, for each $k \in \nn$,
\[
	B_{q,k} := \frac{\mathsf{f}_{q,k}}{d^k}  = \frac{d(n-1) - n q^k(d-1)}{d-n},
\]
where $\mathsf{f}_{q,k}$ is the Frobenius number of the numerical monoid $d^k M_{q,k}$. When $q \in (0,1) \cap \qq$, the limit of the sequence $(B_{q,k})_{k \ge 1}$ exists and is denoted $B_{q,\infty}$:
\[
	B_{q,\infty} = \lim_{k \to \infty} B_{q,k} = \frac{d(n-1)}{d-n}.
\]
In addition, we have already established the following lemma.

\begin{lem} \cite{CM24} \label{lem:Frobenius bound copy}
	For $q \in \qq_{>0} \setminus (\nn \cup \nn_{\ge 2}^{-1})$, we let $S_q$ be the rational monogenic semidomain generated by~$q$. Then the following statements hold.
	\begin{enumerate}
		\item For each $k \in \nn$, if $c/\mathsf{d}(q)^k > B_{q,k}$ for some $c \in \nn$ then $c/\mathsf{d}(q)^k \in S_q$.
		\smallskip
		
		\item If $q \in (0,1)$ then, for each $k \in \nn$, if $c/\mathsf{d}(q)^k \ge B_{q,\infty}$ for some $c \in \nn$ then $c/\mathsf{d}(q)^k \in S_q$.
	\end{enumerate} 
\end{lem}

In order to determine the positive rational parameters $q$ such that $S_q^*$ is root-closed, it is convenient to prove the following proposition.

\begin{proposition} \label{prop:elems in root closure}
	For $q \in \qq_{> 0} \setminus (\nn \cup \nn_{\ge 2}^{-1})$, let $S_q$ be the rational monogenic semidomain generated by~$q$. Then the following statements hold.
	\begin{enumerate}
		\item If $q < 1$ then $m/\mathsf{d}(q)^k \in \overline{S^*_q}$ for all $k,m \in \nn$ such that $m/\mathsf{d}(q)^k > 1$.
		\smallskip
		
		\item If $q > 1$ then $m/\mathsf{d}(q)^k \in \overline{S^*_q}$ for all $k,m \in \nn$ such that $m/\mathsf{d}(q)^k > q^k$.
	\end{enumerate}
\end{proposition}

\begin{proof}
	Fix a positive rational $q$ such that $q \notin \nn \cup \nn_{\ge 2}^{-1}$, and then set $n := \mathsf{n}(q)$ and $d := \mathsf{d}(q)$.
	\smallskip
	
	(1) First, we assume that $q<1$. Take $k,m \in \nn$ such that $r := m/d^k > 1$. Since $r>1$, we can take $j \in \nn$ sufficiently large so that $r^j \ge (n-1)/(1-q)$. As
	\[
		\frac{m^j}{d^{kj}} = r^j \ge \frac{n-1}{1-q} = \frac{d(n-1)}{d-n} = B_{q,\infty},
	\]
	part~(2) of Lemma~\ref{lem:Frobenius bound copy} ensures that $r^j = m^j/d^{kj} \in S_q$. As a consequence, $r$ belongs to the root closure of $S_q^*$, as desired.
	\smallskip
	
	(2) Let us assume now that $q>1$. Take $k,m \in \nn$ such that $r := m/d^k > q^k$. Therefore $r/q^k > 1$ and so one can take $j \in \nn$ large enough so that
	\[
		\left(\frac{r}{q^k}\right)^j>\frac{n(d-1)}{n-d}.
	\]
	Then
	\[
		\frac{m^j}{d^{kj}} = r^j>\frac{n(d-1)}{n-d}q^{jk} > \frac{n(d-1)q^{jk}-d(n-1)}{n-d} = B_{q,jk}.
	\]
	Therefore it follows from part~(1) of Lemma~\ref{lem:Frobenius bound copy} that $r^j = m^j/d^{kj} \in S_q$. As a consequence, $r$ belongs to the root closure of the multiplicative monoid~$S_q^*$, which concludes our proof.
\end{proof}

Let us take a look at an example.

\begin{exam}
	For $q := 2/3$, we consider the rational monogenic semidomain $S_q$. First, notice that $B_{q,\infty} = 3(2-1)/(3-2) = 3$. Now set $r:=11/9$ and observe that $r \notin S_q$: indeed, if $r \in S_q$ then Lemma~\ref{lem:height characterization} would force $r$ to have a polynomial representation of degree at most $2$, and so there would be $c_0, c_1, c_2 \in \nn_0$ such that
	\[
		\frac{11}{9}=c_0+c_1\frac23+c_2\frac49
	\]
		or, equivalently, $11=9c_0+6c_1+4c_2$, which is not possible: indeed, any potential solution $(c_0, c_1, c_2)$ would have $c_0 \in \{0,1\}$ but no pair $(c_1, c_2) \in \nn_0^2$ satisfies that $11 = 6c_1 + 4c_2$ or $2 = 6c_1 + 4c_2$. On the other hand, $r$ belongs to the root closure of $S_q^*$: indeed, $r^6 = 11^6/3^{12} > 3 = B_{q,\infty}$, so Corollary~\ref{cor:Frobenius bound for q<1} gives $r^6 \in S_q^*$. Hence $S_{2/3}$ is not root-closed.
	\hfill $\blacksquare$
\end{exam}

We are in a position to determine the rational monogenic semidomains that are root-closed.

\begin{thm} \label{thm:root-closed cyclic semiring}
	For each $q \in \qq_{>0}$, let $S_q$ be the rational monogenic semidomain generated by $q$. Then the multiplicative monoid of the semidomain $S_q$ is root-closed if and only if $q \in \nn \cup \nn_{\ge 2}^{-1}$.
\end{thm}

\begin{proof}
	Fix a positive rational $q$ and then set $n := \mathsf{n}(q)$ and $d := \mathsf{d}(q)$.
	\smallskip
	First, we observe that the reverse implication follows immediately once we combine Theorem~\ref{thm:UF/HF} with the well-known fact that every UFM is root-closed.
	\smallskip

	For the direct implication, assume that $q \notin \nn \cup \nn^{-1}$, which means that $\min\{n,d\} \ge 2$. Let us argue that $S_q^*$ is strictly contained in $\overline{S_q^*}$. To do so, fix $k \in \nn$ such that $\min\{n^k,d^k\} > 2$. It is convenient to split the rest of the proof into the following two cases.
	\smallskip
	
	\noindent \textsc{Case 1:} $q<1$. In this case, $d > n \ge 2$. As $d \ge 3$, we see that $d \nmid d^k + 1$ and $d \nmid d^k + 2$. On the other hand, the fact that $d^k + 1$ and $d^k + 2$ are consecutive integers ensures that either $n \nmid d^k + 1$ or $n \nmid d^k + 2$. Now set $r := m/d^k$, where $m \in \{d^k + 1, d^k + 2\}$ and $n \nmid m$. Observe that $d \nmid m$ and $m/d^k > 1$. Therefore part~(1) of Proposition~\ref{prop:elems in root closure} ensures that $r = m/d^k \in \overline{S_q^*}$.
	
	Now assume, towards a contradiction, that $r \in S_q$. As $r \in \{1 + 1/d^k, 1 + 2/d^k\}$, we see that $r \notin M_{q,k-1}$, whence the height of $r$ is at least $k$. On the other hand, it follows from Lemma~\ref{lem:height characterization} that the height of $r$ is at most $k$. Hence $h(r) = k$. Thus, $r \in M_{q,k}$ but $r \notin M_{q,k-1}$. Because $n \nmid m$, any factorization of~$r$ in the additive monoid $(S_q,+)$ must contain the additive atom~$1$, while the fact that $r \notin M_{q,k-1}$ ensures that any factorization of~$r$ in the additive monoid $M_{q,k}$ must contain a copy of the atom $q^k$. Hence
	\[
		d^k + n^k = d^k(1 + q^k) \le d^kr = m \le d^k + 2.
	\] 
	Therefore $\min\{n^k, d^k\} = n^k \le 2$, which contradicts our choice of~$k$. Hence $r \notin S_q$, and so $S_q^*$ is strictly contained in $\overline{S_q^*}$.
	\smallskip
	
	\noindent \textsc{Case 2:} $q>1$. In this case, $n > d \ge 2$, and so $n \ge 3$. Since $n^k + 1$ and $n^k + 2$ are consecutive integers, there exists $m \in \{n^k + 1, n^k + 2\}$ such that $d \nmid m$. Also, $n \nmid m$. Set $r := m/d^k$. As $r > n^k/d^k = q^k$, part~(2) of Proposition~\ref{prop:elems in root closure} guarantees that $r \in \overline{S_q^*}$.
	
	As in the previous case, after assuming, towards a contradiction, that $r \in S_q$, one can similarly obtain that $h(r)=k$ and so that any factorization of~$r$ in the additive monoid $M_{q,k}$ must contain a copy of the atom~$1$ and a copy of the atom~$q^k$, whence $d^k + n^k \le d^kr = m \le n^k + 2$. Therefore $\min\{n^k, d^k\} = d^k \le 2$, which contradicts our choice of~$k$. Thus, $r \notin S_q$ and so $S_q^*$ is strictly contained in $\overline{S^*_q}$.
\end{proof}

As a consequence of Theorem~\ref{thm:root-closed cyclic semiring}, we are able to determine the rational parameters $q$ for which the semidomain $S_q$ has the Krull property.

\begin{cor}
	For $q \in \qq_{>0}$, let $S_q$ be the rational monogenic semidomain generated by~$q$. Then the following conditions are equivalent.
	\begin{enumerate}
		\item[(a)] $S_q$ is a Krull semidomain.
		\smallskip
		
		\item[(b)] $q \in \nn \cup \nn_{\ge 2}^{-1}$.
		\smallskip
		
		\item[(c)] $S_q$ is a UFS.
	\end{enumerate}
\end{cor}

\begin{proof}
	(a) $\Rightarrow$ (b): First, assume that $S_q$ is a Krull semidomain or, equivalent, that the multiplicative monoid $S_q^*$ is a Krull monoid. Then $S_q^*$ is a completely integrally closed monoid, and so it is root-closed. Therefore $q \in \nn \cup \nn_{\ge 2}^{-1}$.
	\smallskip
	
	(b) $\Rightarrow$ (c): This is Lemma~\ref{lem:if q is integer or unit fraction S_q is UFS}
	\smallskip
	
	(c) $\Rightarrow$ (a): This follows straightforwardly as UFMs can be characterized as Krull monoids with trivial class group.
\end{proof}

\bigskip
\section*{Acknowledgments} 

During the preparation of this paper, the authors were part of the first CrowdMath Internship (CMI) at MIT, and they would like to thank the program for making this collaboration possible. The authors thank Dr. Marly Gotti for the careful review and for providing helpful feedback on earlier drafts of this paper. Finally, the second author kindly acknowledges support from the NSF under the award DMS-2213323.

\bigskip

\end{document}